\documentclass[11pt]{amsart}
\usepackage{a4wide,amsfonts}
\usepackage{amsthm}
\usepackage{amsmath}
\usepackage{mathrsfs} %mathscr{A}
\usepackage{amssymb}
\usepackage{amscd}
\usepackage{multirow}
\usepackage[all]{xy}

\usepackage{hyperref}

\numberwithin{equation}{section}
\usepackage{epstopdf}
\usepackage{enumerate}

\newtheorem{theorem}{Theorem}[section]
\newtheorem{lemma}[theorem]{Lemma}

\newtheorem{corollary}[theorem]{Corollary}
\newtheorem{proposition}[theorem]{Proposition}
\newtheorem{claim}[theorem]{Claim}

\theoremstyle{definition}
\newtheorem{definition}[theorem]{Definition}
\newtheorem{example}[theorem]{Example}
\newtheorem{notation}[theorem]{Notation}
\newtheorem{remark}[theorem]{Remark}

\newcommand{\dCF}{C_{{\downarrow}\mathsf F}}
\newcommand{\Cl}{\mathrm{Cl}}
\newcommand{\IR}{\mathbb R}
\newcommand{\IN}{\mathbb N}
\newcommand{\w}{\omega}
\newcommand{\K}{\mathcal K}
\newcommand{\F}{\mathcal F}
\newcommand{\Ra}{\Rightarrow}
\newcommand{\U}{\mathcal U}

\newcommand{\MOP}{\mathsf{MOP}}
\newcommand{\DMOP}{\mathsf{DMOP}}
\newcommand{\WDMOP}{\mathsf{WDMOP}}
\newcommand{\GKF}{\mathsf{G_{KF}}}
\newcommand{\GEN}{\mathsf{G_{EN}}}
\newcommand{\SGEN}{\dot{\mathsf G}_{\mathsf{EN}}}

\newcommand{\e}{\varepsilon}

\newcommand{\Lim}{\mathrm{Lim}}

\begin{document}

\title[Function spaces with the Fell hypograph topology]{Baire category properties of function spaces with the Fell hypograph topology}

\author{Leijie Wang}
\address{department of Mathematics, Shantou University, Shantou, Guangdong,
515063, PR China}
\email{16ljwang@stu.edu.cn}
\thanks{}

\author{Taras Banakh}
\address{Jan Kochanowski University in Kielce and
  Ivan Franko National University in Lviv}
\curraddr{}
\email{t.o.banakh@gmail.com}
\thanks{}

%    General info

\date{}

\dedicatory{}


\begin{abstract} For a Tychonoff space $X$ and a subspace  $Y\subset\IR$, we study Baire category properties of the space $\dCF(X,Y)$ of continuous functions from $X$ to $Y$, endowed with the Fell hypograph topology. We characterize pairs $X,Y$ for which the function space $\dCF(X,Y)$ is $\infty$-meager, meager, Baire, Choquet, strong Choquet, (almost) complete-metrizable or (almost) Polish.
\end{abstract}

\maketitle

\tableofcontents

\section{Introduction and Main Results}

In this paper we study Baire category properties of function spaces $\dCF(X,Y)$ and answer a problem, posed by McCoy and Ntantu in \cite{McN}.

For a topological space $X$, the {\em Fell topology} on the space $\Cl(X)$ of all closed subsets of $X$ is generated by the subbase consisting of the sets $$U^-=\{F\in\Cl(X):F\cap U\ne\emptyset\}\mbox{ and }(X\setminus K)^+=\{F\in\Cl(X):F\subset X\setminus K\}$$where $U$ and $K$ run over open and compact sets in $X$, respectively.  The space $\Cl(X)$ endowed with the Fell topology is denoted by $\Cl_{\mathsf F}(X)$.

For a topological space $X$ and a subspace $Y\subset\IR$ of the real line,  let $C(X,Y)$ denote the set of continuous functions from $X$ to $Y$. Identifying each function $f\in C(X,Y)$ with its {\em hypograph} ${\downarrow}f:=\{(x,y)\in X\times \IR:y\le f(x)\}$, we identify $C(X,Y)$ with the subset $\{{\downarrow}f:f\in C(X,Y)\}$ of the hyperspace $\Cl_{\mathsf F}(X\times \IR)$. The topology  on the function space $C(X,Y)$, inherited from the hyperspace $\Cl_{\mathsf F}(X\times \IR)$, is called {\em the Fell hypograph topology}. Let $\dCF(X,Y)$ be the function space $C(X,Y)$ endowed with the Fell hypograph topology.

Repeating the argument of the proof of Lemma~2.1 in \cite{McN}, it can be shown that for any Hausdorff space $X$ and any subspace $Y\subset\IR$, the Fell hypograph topology on $C(X,Y)$ is generated by the subbase consisting of the sets $$
%\begin{aligned}
\lceil U;y\rfloor:=\{f\in C(X,Y):\sup f(U)>y\}\mbox{ and }
\lceil K;y\rceil:=\{f\in C(X,Y):\max f(K)<y\}
%\end{aligned}
$$where $U$ is a non-empty open set in $X$, $K$ is a non-empty compact subset of $X$, and $y\in \IR$.

This description implies that the Fell hypograph topology on $C(X,Y)$ is weaker than the compact-open topology, which is generated by the subbase consisting of the sets $$
[K;U]:=\{f\in C(X,Y):f(K)\subset U\}
%\begin{aligned}
%\lfloor K;a\rfloor:=\{f\in C(X,Y):\min K>a\}\mbox{ \ and \ }\lceil K;a\rceil:=\{f\in C(X,Y):\max K<a\},
%\end{aligned}
$$where $K$ is a compact set in $X$ and $U$ is an open set in $Y$.
The function space $C(X,Y)$ endowed with the compact-open topology will be denoted by $C_k(X,Y)$.

%If $Y=\IR$, then the function spaces $\dCF(X,Y)$ and $C_k(X,Y)$ are denoted by $\dCF(X)$ and $C_k(X)$, respectively.
Topological properties of the function spaces $\dCF(X,Y)$ were
studied in \cite{McN}, \cite{Yang1},
\cite{Yang2},
\cite{Yang3},
\cite{Yang4},
\cite{Yang5},
\cite{Yang6}.% In particular, in Theorem 2.7 \cite{McN} McCoy and Ntantu proved that for any Tychonoff space $X$ the function space $\dCF(X,\IR)$ is topologically homogeneous, which implies that $\dCF(X,\IR)$ is Baire if and only if $\dCF(X,\IR)$ is not meager.

In this paper we shall explore Baire category properties of the function spaces $\dCF(X,Y)$. Let us recall that a topological space $X$ is
\begin{itemize}
\item {\em Baire} if the intersection $\bigcap_{n\in\w}U_n$ of any sequence $(U_n)_{n\in\w}$ of open dense subsets of $X$ is dense in $X$;
\item {\em meager} if $X$ can be written as the countable union of (closed) nowhere dense subsets.
\end{itemize}
It is well-known \cite[8.1]{Ke} that a topological space $X$ is Baire if and only if each non-empty open subspace of $X$ is not meager, and similarly a topological space $X$ is meager if and only if each non-empty open subspace of $X$ is not Baire.

By the classical theorem of Oxtoby \cite{Oxtoby}, Baire spaces can be characterized as topological spaces $X$ in which the player $\mathsf E$ does not have a winning strategy in the Choquet game $\GEN(X)$. The game $\GEN(X)$ is played by two players, $\mathsf E$ and $\mathsf N$ (abbreviated from $\mathsf{Empty}$ and $\mathsf{Non}$-$\mathsf{Empty}$). The player $\mathsf E$ starts the game selecting a non-empty open set $U_1\subset X$. Then the player $\mathsf N$ responds selecting a non-empty open set $V_1\subset U_1$. At the $n$-th inning the player $\mathsf E$ selects a non-empty open set $U_n\subset V_{n-1}$ and player $\mathsf N$ responds selecting a non-empty open set $V_n\subset U_n$. At the end of the game the player $\mathsf E$ is declared the winner if the intersection $\bigcap_{n\in\IN}U_n=\bigcap_{n\in\IN}V_n$ is empty. Otherwise the player $\mathsf N$ wins the game.

We shall be also interested in a variation $\SGEN(X)$ of the Choquet game, called the strong Choquet game. This game is played by two players, $\mathsf E$ and $\mathsf N$. The player $\mathsf E$ starts the game selecting an open set $U_1\subset X$ and a point $x_1\in U_1$. Then the player $\mathsf N$ responds selecting an open neighborhood $V_1\subset U_1$ of $x_1$. At the $n$-th inning the player $\mathsf E$ selects an open set $U_n\subset V_{n-1}$ and a point $x_n\in U_n$ and player $\mathsf N$ responds selecting a neighborhood $V_n\subset U_n$ of $x_n$. At the end of the game the player $\mathsf E$ is declared the winner if the intersection $\bigcap_{n\in\IN}U_n=\bigcap_{n\in\IN}V_n$ is empty. Otherwise the player $\mathsf N$ wins the game.

A topological space $X$ is called
\begin{itemize}
\item {\em Choquet} if the player $\mathsf N$ has a winning strategy in the Choquet game $\GEN(X)$;
\item {\em strong Choquet} if the player $\mathsf N$ has a winning strategy in the strong Choquet game $\SGEN(X)$;
\item {\em Polish} if it is homeomorphic to a separable complete metric space;
\item {\em almost Polish} if it contains a dense Polish subspace;
\item {\em complete-metrizable} if it is homeomorphic to a complete metric space;
\item {\em almost complete-metrizable} if its contains a dense complete-metrizable subspace.
\end{itemize}
For every topological space we have the implications
$$\xymatrix{
\mbox{Polish}\ar@{=>}[d]\ar@{=>}[r]&\mbox{complete-metrizable}\ar@{=>}[d]\ar@{=>}[r]&\mbox{strong Choquet}\ar@{=>}[d]&\mbox{non-meager}\\
\mbox{almost Polish}\ar@{=>}[r]&\mbox{almost complete-metrizable}\ar@{=>}[r]&\mbox{Choquet}\ar@{=>}[r]&\mbox{Baire}\ar@{=>}[u]
}
$$
By \cite[8.16, 8.17]{Ke}, a metrizable separable space is
\begin{itemize}
\item  complete-metrizable if and only if it is strong Choquet;
\item almost complete-metrizable if and only if it is Choquet.
\end{itemize}

 %By (the proof of) Theorem 3.5 in \cite{McN},  for a Tychonoff space $X$ and a convex subset $Y\subset\IR$ containing more than one point, the function space $\dCF(X,Y)$ is Hausdorff if and only if $X$ is {\em almost locally compact} in the sense that $X$ contains a dense locally compact subspace. Observe that a regular space $X$ is almost locally compact if and only if the set $X'$ of points $x\in X$ at which $X$ is not locally compact is nowhere dense in $X$.
In \cite[5.2]{McN} McCoy and Ntantu proved that for a Tychonoff space $X$ the function space $\dCF(X,\IR)$ is complete-metrizable if and only if $\dCF(X,\IR)$ is Polish if and only if $X$ is countable and discrete.

In \cite[5.3]{McN} McCoy and Ntantu posed a problem of characterization of Tychonoff spaces $X$ for which the function space $\dCF(X,\IR)$ is Baire. In Corollary~\ref{c1} we shall prove that this happens if and only if the space $X$ is discrete if and only if the space $\dCF(X,\IR)$ is (strong) Choquet. Then we shall consider a more difficult problem of detecting Baire and (strong) Choquet spaces among function spaces $\dCF(X,Y)$ where $Y$ is a subset of the real line with $\inf Y\in Y$ and $X$ is a $Y$-separated space.

\begin{definition} Let $Y$ be a topological space. A topological space $X$ is defined to be {\em $Y$-separated} if for any distinct points $x_1,x_2\in X$ and any points $y_1,y_2\in Y$ there exists a continuous map $f:X\to Y$ such that $f(x_i)=y_i$ for every $i\in\{1,2\}$.
\end{definition}

A topological space $X$ is called
\begin{itemize}
\item {\em functionally Hausdorff} if it is $[0,1]$-separated;
\item {\em totally disconnected} if it is $\{0,1\}$-separated.
\end{itemize}

It is easy to see that
\begin{itemize}
\item for a connected subspace $Y\subset \IR$ containing more than one point, a topological space $X$ is $Y$-separated if and only if $X$ is functionally Hausdorff;
\item for a disconnected subspace $Y\subset\IR$ a topological space $X$ is $Y$-separated if and only if $X$ is totally disconnected.
\end{itemize}

\begin{theorem}\label{t:BY0} Let $Y\subset \IR$ be a subset with $\inf Y\notin Y$. For any $Y$-separated space $X$, the function space $\dCF(X,Y)$ is
\begin{enumerate}
\item[\textup{(1)}]  {\em Baire} if and only if  $X$ is discrete and the space $Y$ is Baire;
\item[\textup{(2)}]  {\em Choquet} if and only if the space $X$ is discrete and the space $Y$ is almost Polish;
\item[\textup{(3)}]  {\em strong Choquet} if and only if the space $X$ is discrete and the space $Y$ is Polish;
\item[\textup{(4)}] {\em almost complete-metrizable} if and only if $\dCF(X,Y)$ is {\em almost Polish} if and only if $X$ is countable and discrete and the space $Y$ is almost Polish;
\item[\textup{(5)}] {\em complete-metrizable} if and only if $\dCF(X,Y)$ is {\em Polish} if and only if $X$ is countable and discrete and the space $Y$ is Polish.
\end{enumerate}
\end{theorem}

The statements (1)--(5) of Theorem~\ref{t:BY0} are proved in Lemmas~\ref{l:Bnon}, \ref{l:C3}, \ref{l:sC3}, \ref{l:AP2}, \ref{l:P0}, respectively. Taking into account that the real line is a Polish space with $\inf \IR=-\infty\notin\IR$, we conclude that Theorem~\ref{t:BY0} implies the following characterization that  answers  Problem 5.3 \cite{McN} of McCoy and Ntantu.

\begin{corollary}\label{c1} For a Tychonoff space $X$, the following conditions are equivalent:
\begin{enumerate}
\item[\textup{(1)}] $\dCF(X,\IR)$ is Baire;
\item[\textup{(2)}] $\dCF(X,\IR)$ is Choquet;
\item[\textup{(3)}] $\dCF(X,\IR)$ is strong Choquet;
\item[\textup{(4)}] $X$ is discrete.
\end{enumerate}
\end{corollary}

Now, we present a characterization of Baire and Choquet spaces among function spaces $\dCF(X,Y)$ where $Y$ is a subset of the real line with $\inf Y\in Y$.

This characterization involves the Discrete Moving Off Property and Winning Discrete Moving Off Properties (abbreviated by $\DMOP$ and $\WDMOP$), which were introduced and studied by the authors in \cite{BW}. The Discrete Moving Off Property is a modification of  $\MOP$, the Moving Off Property of Gruenhage and Ma \cite{GMa}.

A point $x$ of a topological space $X$ is called {\em isolated} if its singleton $\{x\}$ is clopen set in $X$ (which means that $\{x\}$ is closed-and-open in $X$).

\begin{notation}
For a topological space $X$ let
\begin{itemize}
\item $\dot X$ be the (open) set of all isolated points of $X$,
\item $X'$ be the (closed) set of non-isolated points in $X$,
\item $X'^\circ$ be the interior of the set $X'$ in $X$;
\item $\overline{X'^\circ}$ be the closure of the set $X'^\circ$ in $X$.
\end{itemize}
\end{notation}

A family $\mathcal F$ of subsets of a topological space $X$ is called
\begin{itemize}
\item {\em discrete} if each point $x\in X$ has a neighborhood $O_x\subset X$ that meets at most one set $F\in\F$;
\item a {\em moving off family} if for any compact subset $K\subset X$ there is a non-empty set $F\in\mathcal F$ with $F\cap K=\emptyset$.
\end{itemize}
It is clear that each discrete infinite family is moving off.

\begin{definition} A topological space $X$ is defined to have the {\em Discrete Moving Off Property} (abbreviated $\DMOP$) if any moving off family $\F$ of  finite subsets in $\dot X$ contains an infinite subfamily $\mathcal D\subset\mathcal F$, which is discrete in $X$.
\end{definition}

By \cite{BW},  a topological space $X$ has $\DMOP$ if and only if the player $\mathsf F$ does not have the winning strategy in the infinite game $\GKF(X)$, played by two players, $\mathsf{K}$ and $\mathsf{F}$ according to the following rules. The player $\mathsf{K}$ starts the game. At the $n$-th inning the player $\mathsf K$ chooses a compact subset $K_n\subset X$  and the player $\mathsf F$ responds by choosing a finite subset $F_n\subset \dot X\setminus K_n$. At the end of the game, the player $\mathsf K$ is declared the winner if the indexed family $(F_n)_{n\in\IN}$ is discrete in $X$ (which means that each point $x\in X$ has a neighborhood $O_x\subset X$ that meets at most one set $F_n$); otherwise the player $\mathsf F$ wins the game.

\begin{definition}A topological space $X$ is defined
\begin{itemize}
\item to have the {\em Winning Discrete Moving Off Property} (abbreviated $\WDMOP$) if the player $\mathsf K$ has a winning strategy in the game $\GKF(X)$;
\item to be a {\em $\dot\kappa$-space} if a subset $D\subset\dot X$ is closed in $X$ if and only if for every compact set $K\subset X$ the intersection $D\cap K$ is finite;
\item to be a {\em $\dot\kappa_\w$-space} if there exists a countable family $\K$ of compact subsets of $X$ such that $\dot X\subset\bigcup\K$ and a subset $D\subset\dot X$ is closed in $X$ if and only if for every  $K\in\K$ the intersection $D\cap K$ is finite.
\end{itemize}
\end{definition}
By \cite[6.2]{BW}, every $\dot \kappa_\w$-space has $\WDMOP$, and $\WDMOP$ implies $\DMOP$.

These properties have nice characterizations in terms of the Baire category properties of the function space $$C_k'(X,2)=\big\{f\in C_k(X,\{0,1\}):f(X')\subset\{0\}\big\}.$$

The following theorem was proved in \cite{BW}.

\begin{theorem} For a topological space $X$ the function space $C_k'(X,2)$ is
\begin{itemize}
\item discrete iff $\dot X\subset K$ for some compact set $K\subset X$;
\item complete-metrizable iff $X$ is a $\dot \kappa_\w$-space;
\item Polish iff  $X$ is a $\dot \kappa_\w$-space and the set $\dot X$ is countable;
\item Choquet  iff $X$ has $\WDMOP$;
\item Baire  iff $X$ has $\DMOP$;
\item meager  iff $X$ does not have $\DMOP$.
\end{itemize}
\end{theorem}

Also we need the notion of a $Y$-separable space, which is defined with the help of the $Y$-topology.

For topological spaces $X,Y$, the {\em $Y$-topology} on $X$ is the weakest topology in which all maps $f\in C(X,Y)$ remain continuous. This topology is generated by the subbase consisting of the sets $f^{-1}(U)$ where $f\in C(X,Y)$ and $U$ is an open set in $Y$. Observe that a topological space $X$ is Tychonoff (and zero-dimensional)  if and only if its topology coincides with the $\IR$-topology (and with the $\{0,1\}$-topology).

For a subset $A$ of a topological space $X$ by $\overline{A}^Y$ we denote the closure of $A$ in the $Y$-topology of $X$ and call $\overline{A}^Y$ the {\em $Y$-closure} of $A$. It is clear that the closure $\overline{A}$ of any set $A\subset X$ is contained in its $Y$-closure $\overline{A}^Y$.

\begin{definition} A topological space $X$ is defined to be {\em $Y$-separable} if $X$  contains a meager $\sigma$-compact subset $M\subset X$ such that $X'= \overline{M}^Y$.
\end{definition}

Observe that a topological space $X$ is $Y$-separable if its set $X'$ of non-isolated points is separable (in the standard sense). In Lemma~\ref{l:P1}(2) we shall prove that  a $Y$-separated topological space $X$ is $Y$-separable if the function space $\dCF(X,Y)$ has a countable network.
\smallskip

A topological space $X$ is called {\em Polish+meager} if it contains a Polish subspace $P\subset X$ whose complement $X\setminus P$ is meager in $X$.
It is easy to see that a Polish+meager space is Baire if and only if it is almost Polish.
It is known \cite[8.23]{Ke} that each Borel subset of a Polish space is Polish+meager. A subset $A$ of a topological space $X$ is {\em sequentially closed} if $A$ contains the limit point of any sequence $\{a_n\}_{n\in\w}\subset A$ that converges in $X$.

\begin{theorem}\label{t:main}  Let $Y\subset\IR$ be a Polish+meager subspace  with $\inf Y\in Y\ne\{\inf Y\}$. For a $Y$-separable $Y$-separated space $X$, the function space $\dCF(X,Y)$ is
\begin{enumerate}
\item[\textup{(1)}]  Baire if and only if the space $Y$ is Baire, the set $\dot X$ is dense in $X$, and the space $X$ has $\DMOP$;
\item[\textup{(2)}]  Choquet if and only if the space $Y$ is almost Polish, the set $\dot X$ is dense in $X$, and the space $X$ has $\WDMOP$;
\item[\textup{(3)}]  strong Choquet if (and only if) the space $Y$ is Polish and the set $\dot X$ is (sequentially) closed in $X$;
\item[\textup{(4)}]  almost complete-metrizable if and only if the space $Y$ is almost Polish  and $X$ is a $\dot\kappa_\w$-space with dense set $\dot X$ of isolated points;
\item[\textup{(5)}]  almost Polish if and only if the space $Y$ is almost Polish and $X$ is a $\dot\kappa_\w$-space with countable dense set $\dot X$ of isolated points.
\item[\textup{(6)}]  complete-metrizable if and only if $\dCF(X,Y)$ is Polish if and only if the space $Y$ is Polish and $X$ is countable and discrete.
\end{enumerate}
\end{theorem}

The statements (1)--(6) of Theorem~\ref{t:main} are proved in Propositions~\ref{p:Bm2}, \ref{p:C}, \ref{p:sC}, \ref{p:AP5}, \ref{p:AP6}, \ref{p:P}, respectively. The $Y$-separability of the space $X$ is essential and cannot be removed as shown by the following theorem treating Baire category properties of function spaces $\dCF(X,Y)$ on zero-dimensional compact $F$-spaces.

Let us recall that a topological space $X$ is called an {\em $F$-space} if the closures of any disjoint open $F_\sigma$-sets are disjoint. By Theorem 1.2.5 \cite{vM}, for any locally compact $\sigma$-compact non-compact space $X$ the remainder $\beta X\setminus X$ of the Stone-\v Cech compactification of $X$ is an $F$-space. In particular, the remainder $\beta\IN\setminus \IN$ of the Stone-\v Cech compactification of $\IN$ is an $F$-space.

A topological space $X$ is called {\em countably base-compact} if it has a base $\mathcal B$ of the topology such that for any decreasing sequence $\{B_n\}_{n\in\w}\subset \mathcal B$ the intersection $\bigcap_{n\in\w}\bar B_n$ is not empty. It is easy to see that each countably base-compact regular space is strong Choquet. The countable base-compactness is one of Amsterdam properties, discussed by Aarts and Lutzer in \cite[2.1.4]{Lutzer}.

\begin{theorem}\label{t:F} For any compact zero-dimensional $F$-space $X$ and any closed subset $Y\subset \IR$ with $\inf Y\in Y$ the function space $\dCF(X,Y)$ is countably base-compact and strong Choquet.
\end{theorem}

Theorem~\ref{t:F} will be proved in Section~\ref{s6}.
\smallskip

Now we turn to the problem of classification of meager spaces among function spaces $\dCF(X,Y)$. This classification is rather complicated and depends on the interplay between the following 6+7 properties of the spaces $Y$ and $X$.

For a non-empty subset $Y\subset \IR$ we consider the following 6 properties:
\begin{itemize}
\item[$(Y_M)$] $Y$ is meager;
\item[$(Y_B)$] $Y$ is Baire;
\item[$(Y_N)$] $Y$ is neither meager nor Baire;
\vskip2pt
\item[$(Y_0)$] $\inf Y\notin Y$;
\item[$(Y_1)$] $\inf Y\in Y\setminus \dot Y$;
\item[$(Y_2)$] $\inf Y\in \dot Y\subset Y$.
\end{itemize}
For two symbols $L\in\{B,M,N\}$ and $n\in\{0,1,2\}$ we say that the space $Y$ has property $Y_{Ln}$ if $Y$ has the properties $Y_L$ and $Y_n$. So, for example, $Y_{M2}$ means that the space $Y$ is meager and has the smallest element $\inf Y$, which is an isolated point of $Y$. In fact, among all possible 9 combinations of the properties $Y_M,Y_B,Y_N,Y_0,Y_1,Y_2$ we shall need only 6:  $Y_{N0},Y_{N1},Y_{N2},Y_{B0},Y_{B1},Y_{B2}$.

Next, we introduce 7 properties $X_0,X_1,X_2,X_3,X_C,X_B,X_M$ of a topological space $X$ and 12 combinations of these properties (of which we shall need only 6).

For any topological space $X$ consider the following 7 properties:
\begin{itemize}
\item[$(X_0)$] $X'=\emptyset$;
\item[$(X_1)$] $X'\ne\emptyset=X'^\circ$;
\item[$(X_2)$] $\overline{X'^\circ}$ is not empty and compact;
\item[$(X_3)$] $\overline{X'^\circ}$ is not compact;
\vskip2pt
\item[$(X_C)$] $\dot X\subset K$ for some compact set $K\subset X$;
\item[$(X_B)$] $X$ has $\DMOP$ but fails to have the property $(X_C)$;
\item[$(X_M)$] $X$ does not have $\DMOP$.
\end{itemize}

For two symbols $L\in\{C,B,M\}$ and $n\in\{0,1,2,3\}$ we say that the space $X$ has property $X_{Ln}$ if $X$ has the properties $X_L$ and $X_n$. So, for example, $X_{M1}$ means that the space $X$ does not have $\DMOP$ and the set $X'\ne\emptyset$ is nowhere dense in $X$. In fact, among all possible 12 combinations of the properties $X_0,X_1,X_2,X_3,X_C,X_B,X_M$ we shall be interested only in  6: $X_{C0},X_{C1},X_{C2},X_{B0},X_{B1},X_{B2}$.

Finally, let us consider the following three Baire category properties of the function space $\dCF(X,Y)$:
\begin{itemize}
\item[$(M)$] $\dCF(X,Y)$ is meager;
\item[$(B)$] $\dCF(X,Y)$  is Baire;
\item[$(N)$] $\dCF(X,Y)$  is neither meager nor Baire.
\end{itemize}

The following table describes the Baire category properties $M,B,N$ of the function space $\dCF(X,Y)$, where $Y\subset\IR$ is a Polish+meager space containing more than one point and $X$ is a $Y$-separable $Y$-separated topological space.

\begin{center}
Table 1
\vskip5pt

\begin{tabular}{|c||c|c|c|c|c|c| c|}
\hline
 & $Y_M$& $Y_{N0}$ & $Y_{N1}$ & $Y_{N2}$ & $Y_{B0}$ & $Y_{B1}$ & $Y_{B2}$\\
 \hline
 \hline
$X_{C0}$ & M  & N & N & N & B & B & B\\
 \hline
$X_{C1}$ & M & N & N & N & M & B & B\\
 \hline
$X_{C2}$ & M & M & M & N & M & M & N\\
 \hline
$X_{B0}$ & M & M & M & M & B & B & B\\
 \hline
$X_{B1}$ & M & M & M & M & M & B & B\\
 \hline
$X_{B2}$ & M & M & M & M & M & M & N\\
 \hline
$X_{M}$ & M & M & M & M & M & M & M\\
\hline
$X_{3}$ & M & M & M & M & M & M & M\\
\hline
 \end{tabular}
 \end{center}
\smallskip

This table consists of $8\times 7$ statements on the Baire Category properties of the function spaces $\dCF(X,Y)$. The references to lemmas proving these 56 statements will be given in Section~\ref{s:table}. In fact, the meagerness of the function spaces $\dCF(X,Y)$ will be proved in a stronger form of $\infty$-meagerness, defined as follows.

\begin{definition}
A subset $A$ of a topological space $X$ is called
\begin{itemize}
\item {\em $\infty$-dense} in $X$ if for any compact Hausdorff space $K$, the subset $C_k(K,A)=\{f\in C_k(K,X):f(K)\subset A\}$ is dense in $C_k(K,X)$;
\item {\em $\infty$-codense} if its complement $X\setminus A$ is $\infty$-dense in $X$;
\item {\em $\infty$-meager} if $A$ is contained in a countable union of  closed $\infty$-codense subsets of $X$.
\end{itemize}
A topological space $X$ is called
\begin{itemize}
\item {\em $\infty$-meager} if it is $\infty$-meager in itself;
\item {\em $\infty$-comeager} if $X$ contains an $\infty$-dense Polish subspace.
\end{itemize}
\end{definition}
It is easy to see that each closed $\infty$-codense set is nowhere dense, so each $\infty$-meager set is meager. On the other hand, the singleton $\{0\}$ in the real line is nowhere dense but not $\infty$-codense in $\IR$.

It should be mentioned that $\infty$-meager and $\infty$-comeager spaces play an important role in Infinite-Dimensional Topology and enter as key ingredients in many characterization theorems of model infinite-dimensional spaces, see \cite{BCZ}, \cite{BRZ}, \cite{BP},  \cite{vM1}, \cite{vM2}, \cite{Sak}. %The $\infty$-meagerness of $\dCF(X,Y)$ will be used in the paper \cite{BY}, devoted to recognizing the infinite-dimensional structure of metrizable function spaces $\dCF(X,Y)$.

For any topological space we have the implications
$$\mbox{$\infty$-meager $\Ra$ meager $\Ra$ not Baire $\Ra$ not $\infty$-comeager}.$$
By \cite{Ban}, the linear hull of the Erd\H os set $E=\{(x_i)_{i\in\w}\in\ell_2:(x_i)_{i\in\w}\in\mathbb Q^\w\}$ in the separable Hilbert space $\ell_2$ is an example of a meager (pre-Hilbert) space, which is not $\infty$-meager.

\begin{theorem}\label{t:eq} Let $Y$ be a Polish+meager subset $Y\subset \IR$ and $X$ be a $Y$-separable $Y$-separated topological space $X$. The  function space $\dCF(X,Y)$ is meager if and only if $\dCF(X,Y)$ is $\infty$-meager.
\end{theorem}

This theorem will be proved in Section~\ref{s:table}. In Section~\ref{s:dych} we prove an interesting dichotomy for analytic function spaces $\dCF(X,Y)$. A topological space  is called {\em analytic} if it is a continuous image of a Polish space.

\begin{theorem}\label{t:dycho} Let $Y\subset \IR$ be a non-empty Polish subspace with $\inf Y\notin\dot Y$. If for a $Y$-separated topological space $X$ the function space $\dCF(X,Y)$ is analytic, then $\dCF(X,Y)$ is either $\infty$-meager or $\infty$-comeager.
\end{theorem}

Typical examples of sets $Y\subset \IR$ with properties $Y_{B0}$, $Y_{B1}$ and $Y_{B2}$ are the real line $\IR$, the closed interval $[0,1]$ and the doubleton $\{0,1\}$, respectively.

For these spaces the classification given in Table 1 implies the following characterizations.

\begin{corollary} For an $\IR$-separable functionally Hausdorff space $X$, the following statements are equivalent:
\begin{enumerate}
\item[\textup{(1)}] $\dCF(X,\IR)$ is Baire;
\item[\textup{(2)}] $\dCF(X,\IR)$ is not meager;
\item[\textup{(3)}] $\dCF(X,\IR)$ is not $\infty$-meager;
\item[\textup{(4)}] the space $X$ is discrete.
\end{enumerate}
\end{corollary}

\begin{corollary} For an $\IR$-separable functionally Hausdorff space $X$, the following statements are equivalent:
\begin{enumerate}
\item[\textup{(1)}] $\dCF(X,[0,1])$ is Baire;
\item[\textup{(2)}] $\dCF(X,[0,1])$ is not meager;
\item[\textup{(3)}] $\dCF(X,[0,1])$ is not $\infty$-meager;
\item[\textup{(4)}] $X$ has $\DMOP$ and the set $\dot X$ is dense in $X$;
\item[\textup{(5)}] $\dCF(X,[0,1))$ is Baire;
\item[\textup{(6)}] $\dCF(X,[0,1))$ is not meager;
\item[\textup{(7)}] $\dCF(X,[0,1))$ is not $\infty$-meager.
\end{enumerate}
\end{corollary}

On the other hand, the function space $\dCF(X,\{0,1\})$ behaves differently.

\begin{corollary}\label{c:C2} For a $\{0,1\}$-separable totally disconnected space $X$, the following characterizations hold:
\begin{enumerate}
\item $\dCF(X,\{0,1\})$ is Baire if and only if $X$ has $\DMOP$ and $X'^\circ=\emptyset$;
\item $\dCF(X,\{0,1\})$ is meager if and only if $X$ does not have $\DMOP$ or $\overline{X'^\circ}$ is not compact;
\item $\dCF(X,\{0,1\})$ is neither Baire nor meager if and only if $X$ has $\DMOP$ and the set $\overline{X'^\circ}$ is compact and not empty.
\end{enumerate}
\end{corollary}

\begin{remark} Theorem~\ref{t:F} shows that the $\{0,1\}$-separability of space $X$ cannot be removed from the assumptions of Corollary~\ref{c:C2}(3): for the compact $F$-space $X=\beta\IN\setminus \IN$ the function space $\dCF(X,2)$ is Baire but $X$ has $\DMOP$ and $X'^\circ=X'=X$.
\end{remark}

\section{Function spaces  $\dCF(X,Y)$ over $F$-spaces $X$}\label{s6}

In this section we prove Theorem~\ref{t:F}. Given a compact zero-dimensional $F$-space $X$ and a closed subset $Y\subset \IR$ with $\inf Y\in Y$, we need to show that the function space $\dCF(X,Y)$ is countably base-compact and strong Choquet.

In the space $\dCF(X,Y)$ consider the family $\mathcal B$ of all non-empty open sets of the form
$$\lceil\U;a,b]:=\bigcap_{U\in\U}\lceil U;a(U)\rfloor\cap \lceil U;b(U)\rceil,$$
where $\U$ is a finite cover of $X$ by pairwise disjoint clopen sets and $a,b:\U\to\IR$ are two functions.
It follows from $\lceil\U;a,b]\ne\emptyset$ that for every $U\in\U$ the order interval $$\langle a(U),b(U)\rangle_Y:=\{y\in Y:a(U)<y<b(U)\}$$ is not empty and its closure $[a(U),b(U)]_Y$ in $Y$ is compact. It is clear that $[a(U),b(U)]_Y=[\bar a(U),\underline{b}(U)]\cap Y$ for some real numbers $\bar a(U),\underline{b}(U)$ such that $a(U)\le \bar a(U)<\underline{b}(U)\le b(U)$.

It can be shown that $\mathcal B$ is a base of the Fell hypograph topology of $\dCF(X,Y)$. We claim that this base witnesses that the function space $\dCF(X,Y)$ is countably base-compact.

Fix a decreasing sequence $\{\lceil\U_n,a_n,b_n]\}_{n\in\w}\subset\mathcal B$ of basic open sets. Replacing each cover $\U_n$ by a finer disjoint open cover, we can assume that for every $n\in\w$ each set $U\in\U_{n+1}$ is contained in some set $V\in\U_{n}$. Also we loss no generality assuming that $\U_0=\{X\}$.

For any $n\in\w$ and $U\in\U_n$, fix a point $y_n(U)\in\langle a_n(U),b_n(U)\rangle_Y$ and let $\bar a_n(U)$ and $\underline{b}_n(U)$ be two real numbers such that $[a_n(U),b_n(U)]_Y=Y\cap[\bar a_n(U),\underline{b}_n(U)]$.

For any $n\le m$ we can use the inclusion $\lceil\U_m;a_m,b_m]\subset\lceil \U_n;a_n,b_n]$ to show that the following two conditions are satisfied:
\begin{itemize}
\item[(a)] for any $U\in\U_n$ there exists $V\in\U_m$ such that $V\subset U$ and  $y_m(V)\ge \bar a_m(V)\ge \bar a_n(U)$;
\item[(b)] for any $U\in\U_n$ and $V\in\U_m$ with $V\subset U$ we have $y_m(V)\le \underline{b}_m(V)\le\underline b_n(U)$;
\end{itemize} The condition (b) implies that the set $\bigcup_{m\in\w}\{y_m(V):V\in\U_m\}$ is contained in the compact set $[\inf Y,b_0(X)]_Y:=Y\cap [\inf Y,\underline b_0(X)]$.

For every point $x\in X$ let $\Lim(x)$ be the set of all points $y\in Y$ such that for any neighborhoods $O_x\subset X$ and $O_y\subset \IR$ of $y$ the set $\bigcup_{n\in\w}\{U\in\U_n:U\cap O_x\ne\emptyset,\;y_n(U)\in O_y\}$ is infinite.

The compactness of the set $[\inf Y,\underline b_0(X)]_Y\supset\{y_n(U):n\in\w,\;U\in\U_n\}$ implies that the set $\Lim(x)$ is not empty. We claim that $\Lim(x)$ is a singleton. To derive a contradiction, assume that $\Lim(x)$ contains two points $y<z$. Then $$W_-=\bigcup_{n\in\w}\{U\in\U_n:y_n(U)<\tfrac12(y+z)\}\mbox{ and }   W_+=\bigcup_{n\in\w}\{U\in\U_n:y_n(U)>\tfrac12(y+z)\}$$ are two disjoint open $F_\sigma$-sets with $x\in\overline{W_-}\cap\overline{W_+}$, which is not possible in $F$-spaces. This contradiction shows that the set $\Lim(x)$ contains a single point $\lambda(x)\in Y$.

Using the equality $\Lim(x)=\{\lambda(x)\}$ and the compactness of the set $[\inf Y,\underline b_0(X)]_Y$, it is possible to prove that the function $\lambda:X\to [\inf Y,b_0(X)]_Y\subset Y$ is continuous.

It remains to show that $\lambda$ belongs to the closure $\overline{\lceil\U_n;a_n,b_n]}$ of each basic set $\lceil\U_n;
a_n,b_n]$ in $\dCF(X,Y)$. It is easy to see that $$\overline{\lceil\U_n;
a_n,b_n]}=\bigcap_{U\in\U_n}\{f\in\dCF(X,Y):\bar a_n(U)\le \max f(U)\le \underline b_n(U)\}.$$ So, we need to check that $
\bar a_n(U)\le \max \lambda(U)\le \underline b_n(U)$. By the condition (a), for any $m\ge n$ there exists a set $V_m\in\U_m$ such that $y_m(V_m)\ge \bar a_m(V_m)\ge \bar a_n(U)$. By the compactness of $U$, there exists a point $x\in U$ whose any neighborhood $O_x$ intersects infinitely many sets $V_m$. For this point $x$ the value $\lambda(x)$ is contained in the closure of the set $\{y_m(V_m)\}_{m\ge n}\subset [\bar a_n(U),\underline b_0(U)]$. So, $\max\lambda(U)\ge \lambda(x)\ge\bar a_n(U)$.

On the other hand, the condition (b) guarantees that $$\bigcup_{m\ge n}\{y_m(V):V\in\U_m,\;V\subset U\}\subset [\inf Y,\underline b_n(U)],$$ which implies that $\lambda(U)\subset [\inf Y,\underline b_n(U)]$ and finally $$\bar a_n(U)\le \max\lambda(U)\le\underline b_n(U).$$
 This completes the proof of the countable base-compactness of $\dCF(X,Y)$.

By \cite[Theorem 3.7]{McN}, the function space $\dCF(X,Y)$ is regular (since $X$ is compact). Being countably base-compact, the regular space $\dCF(X,Y)$ is strong Choquet.

\section{Separate continuity of the lattice operations on $\dCF(X,Y)$}

Observe that any subset $Y\subset \IR$ is closed under the operations of $\min$ and $\max$. For any topological space $X$, these two operations induce two lattice operations on the function space $\dCF(X,Y)$:
$$\min:\dCF(X,Y)\times \dCF(X,Y)\to\dCF(X,Y),\;\;\min:(f,g)\mapsto \min\{f,g\}$$
and
$$\max:\dCF(X,Y)\times \dCF(X,Y)\to\dCF(X,Y),\;\;\max:(f,g)\mapsto \max\{f,g\},$$
where $\min\{f,g\}:x\mapsto\min\{f(x),g(x)\}$ and $\max\{f,g\}:x\mapsto\max\{f(x),g(x)\}$ for $x\in X$.

\begin{lemma}\label{l:min} For any non-empty set $Y\subset\IR$ and any continuous function $\hbar:X\to Y$ defined on a topological space $X$, the  functions $$\wedge_\hbar:\dCF(X,Y)\to\dCF(X,Y),\;\;\wedge_\hbar:f\mapsto\min\{f,\hbar\},$$
and
$$\vee_\hbar:\dCF(X,Y)\to\dCF(X,Y),\;\;\vee_\hbar:f\mapsto\max\{f,\hbar\},$$
are continuous.
\end{lemma}

\begin{proof} For the continuity of the function $\wedge_\hbar$, it suffices to prove that for any open set $U\subset X$, compact set $K\subset X$ and real number $r$ the preimages $$\wedge_\hbar^{-1}(\lceil U;r\rfloor)\mbox{ \ and \ } \wedge_\hbar^{-1}(\lceil K;r\rceil)$$are open in $\dCF(X,Y)$.

To show that $\wedge_\hbar^{-1}(\lceil U;r\rfloor)$ is open, fix any function $f\in \wedge_\hbar^{-1}(\lceil U;r\rfloor)$. It follows that $\min\{f,\hbar\}\in \lceil U;r\rfloor$ and hence $\min\{f(x),\hbar(x)\}>r$ for some $x\in U$. By the continuity of the functions $f$ and $\hbar$, the point $x$ has an open neighborhood $O_x\subset U$ such that $\inf f(O_x)>r$ and $\inf \hbar(O_x)>r$. Then $\lceil O_x;r\rfloor$ is an open neighborhood of $f$ in $\dCF(X,Y)$ such that $\lceil O_x;r\rfloor\subset \wedge_\hbar^{-1}(\lceil U;r\rfloor)$.

To show that $\wedge_\hbar^{-1}(\lceil K;r\rceil)$ is open, fix any function $f\in \wedge_\hbar^{-1}(\lceil K;r\rceil)$. It follows that $\min\{f,\hbar\}\in \lceil K;r\rceil$. Consider the closed (and thus compact) subset $\tilde K=\{x\in K:\hbar(x)\ge r\}$ of $K$ and observe that $\lceil \tilde K,r\rceil$ is an open neighborhood of $f$, contained in the set $\wedge^{-1}_\hbar(\lceil K,r\rceil)$.
\smallskip

Next, we check that the map $\vee_\hbar$ is continuous. Fix an open set $U\subset X$, a compact set $K\subset X$, and a real number $r$.

To show that $\vee_\hbar^{-1}(\lceil U;r\rfloor)$ is open, fix any function $f\in \vee_\hbar^{-1}(\lceil U;r\rfloor)$. It follows that $\max\{f,\hbar\}\in \lceil U;r\rfloor$ and hence $\max\{f(x),\hbar(x)\}>r$ for some $x\in U$. If $\hbar(x)>r$, then $\vee_\hbar^{-1}(\lceil U;r\rfloor)=\dCF(X,Y)$ is trivially open in $\dCF(X,Y)$.
If $\hbar(x)\le r$, then $f(x)>r$ and then $f\in \lceil U;r\rfloor\subset \vee_\hbar^{-1}(\lceil U;r\rfloor)$ and $f$ is an interior point of $\lceil U;r\rfloor$.

To show that $\vee_\hbar^{-1}(\lceil K;r\rceil)$ is open, fix any function $f\in \vee_\hbar^{-1}(\lceil K;r\rceil)$. It follows that $\max\{f,\hbar\}\in \lceil K;r\rceil$ and hence $\max f(K)<r$ and $\max\hbar(K)<r$. Then $f\in\lceil K;r\rceil\subset \vee_{\hbar}^{-1}(\lceil K;r\rceil)$.
\end{proof}

\section{Extension of functions defined on $Y$-separated spaces}

In this section we establish one helpful extension property of $Y$-separated spaces.

\begin{lemma}\label{l:e} Let $Y\subset\IR$ be a non-empty subspace and $X$ be a $Y$-separated topological space. Any continuous function $f:K\to Y$ defined on a compact subset $K\subset X$ admits a continuous extension $\bar f:X\to Y$.
\end{lemma}

\begin{proof} The conclusion of the lemma is trivially true if $Y$ is a singleton. So, assume that $Y$ contains more than one point. Let $Z=[0,1]$ if $Y$ is connected and $Z=\{0,1\}$ if $Y$ is disconnected. The $Y$-separated property of $X$ implies that the space $X$ is $Z$-separated.

Then for any distinct points $a,b\in K$ we can choose a continuous function $\delta_{a,b}:X\to Z$ such that $\delta_{a,b}(a)\ne \delta_{a,b}(b)$. Let $D=\{(a,b)\in K\times K:a\ne b\}$ and observe that the map
$$\delta:X\to Z^D,\;\;\delta:x\mapsto (\delta_{a,b}(x))_{(a,b)\in D}$$
is continuous  and its restriction $h=\delta{\restriction}K:K\to \delta(K)\subset Z^D$ is injective and hence is a homeomorphism of the compact space $K$ onto $\delta(K)$. The set $\delta(K)\subset Z^D$ is compact and hence closed in the compact Hausdorff space $Z^D$. We claim that the continuous map $$g:\delta(K)\to Y,\;\;g:z\mapsto f\circ h^{-1}(z),$$ admits a continuous extension $\bar g:Z^D\to Y$.

If $Y$ is connected, then this follows from the normality of the compact Hausdorff space $Z^D$ and the Tietze-Urysohn Theorem 2.1.8 in \cite{Eng}.

If $Y$ is disconnected, then the space $Z^D=\{0,1\}^D$ is zero-dimensional, and the continuous map $g:\delta(K)\to g(K)\subset Y$ has a continuous extension $\bar g:Z^D\to g(K)\subset Y$  by Proposition 6.1.10 in \cite{Chig}.

Then $\bar f:=\bar g\circ\delta:X\to Y$ is a required continuous extension of the map $f:K\to Y$.
\end{proof}

\section{The $\infty$-density of some subsets in $\dCF(X,Y)$}%\label{s20}\label{s2}

\begin{lemma}\label{l:d} Let $Y\subset\IR$ be a subset and $X$ be a $Y$-separated space. For any non-empty compact nowhere dense set $K\subset X$ and any real numbers $y<u$ with $y\in Y$ the basic open set $\lceil K;u\rceil$ is $\infty$-dense in $\dCF(X,Y)$.
\end{lemma}

\begin{proof} If $u$ is greater than any element of $Y$, then $\lceil K;u\rceil=\dCF(X,Y)$ and there is nothing to prove. So, we assume that $u\le\bar y$ for some $\bar y\in Y$.

Given any compact Hausdorff space $Z$, we need to prove that the subset $$C_k(Z,\lceil K;u\rceil)=\{f\in C_k(Z,\dCF(X,Y)):f(Z)\subset \lceil K;u\rceil\}$$ is dense in $C_k(Z,\dCF(X,Y))$. Fix any function $\mu\in C_k(Z,\dCF(X,Y))$ and a neighborhood $O_\mu$ of $\mu$ in $C_k(Z,\dCF(X,Y))$. Given a point $z\in Z$, it will be convenient to denote the function $\mu(z)\in\dCF(X,Y)$ by $\mu_z$.

By definition, the Fell-hypograph topology $\mathcal B$ on $\dCF(X,Y)$ has a base $\mathcal B$ consisting of the sets $$\lceil U_1;a_1\rfloor\cap\dots\cap\lceil U_n;a_n\rfloor\cap \lceil K_1;b_1\rceil\cap\dots\cap \lceil K_m;b_m\rceil$$where $U_1,\dots,U_n\subset X$ are non-empty open sets,  $K_1,\dots,K_m\subset X$ are non-empty compact sets, and $a_1,\dots,a_n,b_1,\dots,b_m\in\IR$.

On the other hand, the compact-open topology on the space $C_k(K,\dCF(X,Y))$ is generated by the subbase consisting of the sets $$[Z;B]:=\{f\in\dCF(X,Y):f(Z)\subset B\},$$where $Z$ is a non-empty compact set in $K$ and $B\in\mathcal B$.

So, without loss of generality, we can assume that the neighborhood $O_\mu$ is of basic form $$O_\mu=\bigcap_{i=1}^m[Z_i;B_i]$$for some non-empty compact sets $Z_1,\dots,Z_m\subset Z$ and some basic open sets $B_1,\dots,B_m\in\mathcal B$.

For every $i\le m$ find a non-empty finite family $\U_i$ of non-empty open sets in $X$, a finite family $\K_i$ of non-empty compact sets in $X$, and two functions $a_i:\mathcal U_i\to\IR$ and $b_i:\mathcal K_i\to\IR$ such that $$B_i=\bigcap_{U\in\U_i}\lceil U_i;a_i(U_i)\rfloor\cap\bigcap_{\kappa\in\K_i}\lceil \kappa;b_i(\kappa)\rceil.$$

Let $$a:=\max_{i\le m}\max a_i(\U_i).$$

\begin{claim}\label{cl:s>ay} There exists a point $s\in Y$ such that $s>\max\{a,y\}$.
\end{claim}

\begin{proof} Find $i\le m$ and $U\in\U_i$ such that $a=\max a_i(\U_i)=a_i(U)$. Choose any point $z_i\in Z_i$ and consider the continuous function $\mu_{z_i}=\mu(z_i):X\to Y$. It follows from $z_i\in Z_i$ and $\mu_{z_i}\in [Z_i,B_i]$ that $\mu_{z_i}\in B_i\subset \lceil U;a_i(U)\rfloor$ and hence $\sup \mu_{z_i}(U)>a_i(U)=a$. Then there exists an element $t\in \mu_{z_i}(U)\subset Y$ such that $t>a$. Then the element $s=\max\{t,\bar y\}$ belongs to $Y$ and $s>\max\{a,y\}$.
\end{proof}

 For every $i\le m$ and $z\in Z_i$ consider the function $\mu_z=\mu(z)\in \dCF(X,Y)$ and observe that $\mu\in O_\mu\subset [Z_i;B_i]$ implies $\mu_z\in B_i\subset \bigcap_{U\in\U_i}\lceil U;a_i(U)\rfloor$. Then for every $U\in\U_i$ we can choose a point $x_{z,U}\in U$ such that $\mu_z(x_{z,U})>a_i(U)$. Since the compact set $K$ is nowhere dense in $X$, we can additionally assume that $x_{z,U}\notin K$. Using Lemma~\ref{l:e}, construct a continuous function $\hbar_{z,U}:X\to Y$ such that $\hbar_{z,U}(K)=\{y\}$ and $\hbar_{z,U}(x_{z,U})=s>a\ge a_i(U)$. Then the open set $W_{z,U}:=\{x\in U:\hbar_{z,U}(x)>a_i(U)\}$ is an open neighborhood of the point $x_{z,U}$ and $O_{z,U}:=\mu^{-1}(\lceil W_{z,U};a_i(U)\rfloor)$ is an open neighborhood of $z$ in $Z$.  By the compactness of $Z_i$, there exists a finite set $F_{i,U}\subset Z_i$ such that $Z_i\subset\bigcup_{z\in F_{i,U}}O_{z,U}$. Consider the continuous function $\hbar:X\to Y$, defined by
 $$\hbar=\max\{\hbar_{z,U}:i\le m,\;U\in\U_i,\;z\in F_{i,U}\}$$and observe that $\hbar(K)=\{y\}$.

Lemma~\ref{l:min} implies that the map $\mu':Z\to \dCF(X,Y)$ assigning to each $z\in Z$ the function $\mu_z'=\min\{\mu_z,\hbar\}$ is continuous. Taking into account that $\max \mu_z'(K)\le\max \hbar (K)=y<u$, we conclude that $\mu_z'\in\lceil K;u\rceil$ and hence $\mu'(Z)\subset \lceil K;u\rceil$.

It remains to check that $\mu'\in O_\mu$. Since $O_\mu=\bigcap_{i=1}^m[Z_i;B_i]$, we should prove that for any $i\le n$ and point $z_i\in Z_i$, the function $\mu'_{z_i}=\min\{\mu_{z_i},\hbar\}$ belongs to the set $B_i=\bigcap_{U\in\U_i}\lceil U;a_i(U)\rfloor\cap\bigcap_{\kappa\in\K_i}\lceil\kappa;b_i(\kappa)\rceil$.

Observe that for every  $\kappa\in\K_i$ we have
$\max \mu'_{z_i}(\kappa)\le\max \mu_{z_i}(\kappa)<b_i(\kappa)$. This implies that $\mu'_{z_i}\in\bigcap_{\kappa\in\K_i}\lceil \kappa;b_i(\kappa)\rceil$.

To show that $\mu'_{z_i}\in \bigcap_{U\in\U_i}\lceil U;a_i(U)\rfloor$, take any $U\in\U$ and find $z\in F_{i,U}$ such that $z_i\in O_{z,U}$. By the definition of $O_{z,U}=\mu^{-1}(\lceil W_{z,U};a_i(U)\rfloor)$, we get $\sup \mu_{z_i}(W_{z,U})>a_i(U)$. Consequently, there exists a point $w\in W_{z,U}\subset U$ such that $\mu_{z_i}(w)>a_i(U)$. By the definition of the set $W_{z,U}=\{x\in U:\hbar_{z,U}(x)>a_i(U)\}$, we get $\hbar(w)\ge \hbar_{z,U}(w)>a_i(U)$. Then
$$\sup\mu_{z_i}'(U)\ge \mu'_{z_i}(w)=\min\{\mu_{z_i}(w),\hbar(w)\}>a_i(U)$$and $\mu_{z_i}'\in \lceil U;a_i(U)\rfloor$.

Therefore, $$\mu'_{z_i}\in \bigcap_{U\in\U_i}\lceil U;a_i(U)\rfloor\cap\bigcap_{\kappa\in\K_i}\lceil \kappa;b_i(\kappa)\rceil=B_i$$and we are done.
\end{proof}

\begin{lemma}\label{l:Dni} Let $Y\subset \IR$ be a subspace such that $\inf Y\notin\dot Y$. For any topological space $X$ and any non-empty open subset $V\subset X$ the basic open set $\lceil V,\inf Y\rfloor$ is $\infty$-dense in $\dCF(X,Y)$.
\end{lemma}

\begin{proof} Given any compact Hausdorff space $Z$, we need to prove that the subset $$C_k(Z,\lceil V;\inf Y\rfloor)=\{f\in C_k(Z,\dCF(X,Y)):f(Z)\subset \lceil V;\inf Y\rfloor\}$$ is dense in $C_k(Z,\dCF(X,Y))$. Fix any function $\mu\in C_k(Z,\dCF(X,Y))$ and a neighborhood $O_\mu$ of $\mu$ in $C_k(Z,\dCF(X,Y))$.

We lose no generality assuming that $O_\mu$ is of the basic form $O_\mu=\bigcap_{i=1}^m[Z_i;B_i]$ for some non-empty compact sets $Z_1,\dots,Z_m\subset K$ and some sets $$B_i=\bigcap_{U\in\U_i}\lceil U;a_i(U)\rfloor\cap\bigcap_{\kappa\in\K_i}\lceil\kappa;b_i(\kappa)\rceil,$$where $\U_i$ is a finite family of non-empty open sets in $X$, $\K_i$ is a finite non-empty family of non-empty compact sets in $X$, and $a_i:\mathcal U_i\to\IR$, $b_i:\mathcal K_i\to\IR$ are functions.

Let $b=\min_{i\le m}\min_{\kappa\in\K_i}b_i(\kappa)$. Find $i\le m$ and $\kappa\in K_i$ such that $b=b_i(\kappa)$. The inclusion $\mu\in O_\mu\subset [Z_i;B_i]$ implies that the set $B_i$ is not empty and hence contains some function $\beta:X\to Y$. For this function we get $\beta\in B_i\subset \lceil\kappa;b_i(\kappa)\rceil$ and hence $\inf Y\le\max \beta(\kappa)<b_i(\kappa)=b$. So, $b>\inf Y$. Since the point $\inf Y$ is not isolated in $Y$, there exists an element $y\in Y$ such that $\inf Y<y<b$. Let $\hbar:X\to\{y\}\subset Y$ be the constant function. By Lemma~\ref{l:min}, the function $$\mu':K\to \dCF(X,Y),\;\mu':z\mapsto \mu'_z:=\max\{\mu_z,\hbar\},$$ is continuous. It is easy to see that $\mu'(K)\subset \lceil U;\inf Y\rfloor$ and $\mu'\in O_\mu$.
\end{proof}

\begin{lemma}\label{l:Dnc} Let $Y\subset \IR$ be a subspace such that $\inf Y\in Y$. For any $Y$-separated space $X$, any open subset $V\subset X$ with non-compact closure $\overline{V}$, and any real number $u$ with $\{y\in Y:y>u\}\ne\emptyset$, the basic open set $\lceil V,u\rfloor$ is $\infty$-dense in $\dCF(X,Y)$.
\end{lemma}

\begin{proof} Given any compact Hausdorff space $Z$, we need to prove that the subset $$C_k(Z,\lceil V;u\rfloor)=\{f\in C_k(Z,\dCF(X,Y)):f(Z)\subset \lceil V;u\rfloor\}$$ is dense in $C_k(Z,\dCF(X,Y))$. Fix any function $\mu\in C_k(Z,\dCF(X,Y))$ and a neighborhood $O_\mu$ of $\mu$ in $C_k(Z,\dCF(X,Y))$.

We lose no generality assuming that $O_\mu$ is of the basic form $O_\mu=\bigcap_{i=1}^m[Z_i;B_i]$ for some non-empty compact sets $Z_1,\dots,Z_m\subset K$ and some sets $$B_i=\bigcap_{U\in\U_i}\lceil U;a_i(U)\rfloor\cap\bigcap_{\kappa\in\K_i}\lceil\kappa;b_i(\kappa)\rceil,$$where $\U_i$ is a finite non-empty family of non-empty open sets in $X$, $\K_i$ is a finite non-empty family of non-empty compact sets in $X$, and $a_i:\mathcal U_i\to\IR$, $b_i:\mathcal K_i\to\IR$ are functions.

Let $a=\max_{i\le m}\max_{U\in\U_i}a_i(U_i)$.

\begin{claim} There exists an element $s\in Y$ such that $s>\max\{a,u\}$.
\end{claim}

\begin{proof} If $a\le u$, then take any element $s\in Y$ with $s>u$ and conclude that $s>u=\max\{a,u\}$.

So, we assume that $a>u$. Find $i\le m$ and $U\in\U_i$ with $a=a_i(U)$. Since $\mu\in O_\mu\subset [Z_i;B_i]$, the set $B_i$ is not empty and hence contains some function $\beta\in B_i\subset\lceil U,a_i(U)\rfloor$. Then $\sup \beta(U)>a_i(U)=a$ and hence there exists a point $s\in \beta(U)\subset Y$ with $s>a=\max\{a,u\}$.
\end{proof}

Consider the compact set $\kappa=\bigcup_{i\le m}\bigcup\K_i$. Since the set $V$ has non-compact closure, $V\not\subset\kappa$, so we can choose a point $v\in V\setminus \kappa$. Applying Lemma~\ref{l:e}, find a continuous function $\hbar:X\to Y$ such that $\hbar(\kappa)\subset\{\inf Y\}$ and $\hbar (v)=s$. By Lemma~\ref{l:min}, the map
$$\mu':K\to \dCF(X,Y),\;\mu':z\mapsto \max\{\mu(z),\hbar\},$$ is continuous. It is easy to see that $\mu'(K)\subset\lceil V,u\rfloor$ and $\mu'\in O_\mu$.
\end{proof}

For sets $A\subset X$ and $B\subset Y\subset\IR$ let $$[A;B]:=\{f\in\dCF(X,Y):f(A)\subset B\}.$$ For a point $y\in Y\subset\IR$ we put ${\downarrow}y:=\{u\in Y:u\le y\}$. The following lemma is a modification of Lemma~\ref{l:d}.

\begin{lemma}\label{l:dXy}  For any subset $Y\subset \IR$, real numbers $y<\bar y$ in the set $Y$, and a topological space $X$, the set
$[X';{\downarrow}y]$ is $\infty$-dense in the subspace $[{X'^\circ};{\downarrow}y]$ of $\dCF(X,Y)$.
\end{lemma}

\begin{proof}  Given any compact Hausdorff space $Z$, we need to prove that the subset $$C_k(Z,[X';{\downarrow}y])=\{\mu\in C_k(Z,\dCF(X,Y)):\mu(Z)\subset [X';{\downarrow}y]\}$$ is dense in $C_k(Z,[{X'^\circ};{\downarrow}y])$. Fix any function $\mu\in C_k(Z,[{X'^\circ};{\downarrow}y])$ and a neighborhood $O_\mu$ of $\mu$ in $C_k(Z,[{X'^\circ};{\downarrow}y])$.

We lose no generality assuming that $O_\mu$ is of the basic form $O_\mu=\bigcap_{i=1}^m[{X'^\circ};{\downarrow}y]\cap [Z_i;B_i]$ for some compact sets $Z_1,\dots,Z_m\subset K$ and some sets $$B_i=\bigcap_{U\in\U_i}\lceil U;a_i(U)\rfloor\cap\bigcap_{\kappa\in\K_i}\lceil\kappa;b_i(\kappa)\rceil,$$where $\U_i$ is a non-empty finite family of non-empty open sets in $X$, $\K_i$ is a finite non-empty family of non-empty compact sets in $X$, and $a_i:\mathcal U_i\to\IR$, $b_i:\mathcal K_i\to\IR$ are functions.

Let $a=\max_{i\le m}\max_{U\in\U_i}a_i(U_i)$ and $b=\min_{i\le m}\min_{\kappa\in\K_i}b_i(\kappa)$. Repeating the argument of Claim~\ref{cl:s>ay}, we can find a real number $s\in Y$ such that $s>a$ and $s\ge \bar y>y$.

For every $i\le m$ and $z\in Z_i$ consider the function $\mu_z=\mu(z)\in [{X'^\circ};{\downarrow}y]$ and observe that $\mu\in O_\mu\subset [{X'^\circ};{\downarrow}y]\cap [Z_i;B_i]$ implies $\mu_z\in B_i\subset \bigcap_{U\in\U_i}\lceil U;a_i(U)\rfloor$. Then for every $U\in\U_i$ with $a_i(U)\ge y$, we can choose a point $x_{z,U}\in U$ such that $\mu_z(x_{z,U})>a_i(U)$. Since $a_i(U)\ge y$,  the inclusion $\mu_z\in[{X'^\circ};{\downarrow}y]=[\overline{X'^\circ};{\downarrow}y]$ implies that $x_{z,U}\notin \overline{X'^\circ}$. Since  the set  $X'\setminus X'^\circ\supset X'\setminus \overline{X'^\circ}$ is nowhere dense in $X$, we can replace $x_{z,U}$ by a near point in the set $U\setminus X'$ and additionally assume that $x_{z,U}\in \dot X$.

It follows that  $O_{z,U}:=\mu^{-1}(\lceil\{x_{z,U}\};a_i(U)\rfloor)$ is an open neighborhood of $z$ in $Z$.  By the compactness of $Z_i$, there exists a finite set $F_{i,U}\subset Z_i$ such that $Z_i\subset\bigcup_{z\in F_{i,U}}O_{z,U}$. Consider the finite set
 $$E=\bigcup_{i=1}^m\bigcup_{U\in\U_i}\{x_{z,U}:a_i(U)\ge y,\;\;z\in F_{i,U}\}\subset\dot X$$and define a continuous function $\hbar:X\to Y$ by the formula
 $$\hbar(x)=\begin{cases}s&\mbox{if $x\in E$};\\
  y&\mbox{if $x\in X\setminus E$}.
  \end{cases}
  $$
Lemma~\ref{l:min} implies that the map $\mu':Z\to [{X'^\circ};{\downarrow}y]$ assigning to each $z\in Z$ the function $\mu_z'=\min\{\mu_z,\hbar\}$ is continuous. Taking into account that $\max \mu_z'(X')\le\max \hbar (X')=y$, we conclude that $\mu_z'\in [X';{\downarrow}y]$ and hence $\mu'(Z)\subset [X';{\downarrow}y]$.

By analogy with the proof of Lemma~\ref{l:d}, we can show that $\mu'\in O_\mu$.
\end{proof}

%5\begin{corollary}\label{c:dXy}  For any subset $Y\subset \IR$, real numbers $y<\bar y$ in the set $Y$, and a topological space $X$ with dense $\dot X$ of isolated points, the set
%$[X';{\downarrow}y]$ is $\infty$-dense in the space $\dCF(X,Y)=[{X'^\circ};{\downarrow}y]$.
%\end{corollary}

%\begin{lemma}\label{l:s1} Let $Y$ be a subset of the real line, containing more than one point. If $\inf Y\notin Y$, then for any non-discrete $Y$-separated space $X$, the function space $\dCF(X,Y)$ is $\infty$-meager.
%\end{lemma}

%\begin{proof} Fix any non-isolated point $x_0\in X$. Since $\inf Y\notin Y$, there exists a strictly decreasing sequence $\{y_n\}_{n\in\w}\subset Y$ with $\lim_{n\to\infty}y_n=\inf Y$. By Lemma~\ref{l:d}, for every $n\in\w$ the basic open set $$\lceil \{x_0\};y_n\rceil=\{f\in\dCF(X,Y):f(x_0)<y_n\}$$is $\infty$-dense in $\dCF(X,Y)$. Then the function space $$\dCF(X,Y)=\bigcup_{n\in\w}\dCF(X,Y)\setminus\lceil \{x_0\};y_n\rceil$$is $\infty$-meager, being a countable union of closed $\infty$-codense sets $\dCF(X,Y)\setminus\lceil\{x_0\},y_n)\rceil$, $n\in\w$.
%\end{proof}

\section{The subspace $\dCF'(X,Y)$}

Given a topological space $X$ and a subset $Y\subset\IR$ with $\inf Y\in Y$, consider the subset  $$\dCF'(X,Y):=\big\{f\in \dCF(X,Y):f(X')\subset \{\inf Y\}\big\}=[X';\{\inf Y\}]$$in the function space $\dCF(X,Y)$.

In this section we establish some properties of the subspace $\dCF'(X,Y)$ of $\dCF(X,Y)$. The following lemma is a partial case of Lemma~\ref{l:dXy}.

\begin{lemma}\label{l:C'd}  For any subset $Y\subset \IR$ with $\inf Y\in Y\ne\{\inf Y\}$ and any topological space $X$ the set $\dCF'(X,Y)$ is $\infty$-dense in the subspace $[{X'^\circ};\{\inf Y\}]$ \textup{(}which is equal to $\dCF(X,Y)$ if $X'^\circ=\emptyset\,)$.
\end{lemma}

\begin{lemma}\label{l:Gdelta} If  $X$ is $Y$-separable, then $\dCF'(X,Y)$ is a $G_\delta$-set in $\dCF(X,Y)$.
\end{lemma}

\begin{proof}  Being $Y$-separable, the space $X$ contains a meager $\sigma$-compact set $M$ such that $X'=\overline{M}^Y$. Write $M$ as the countable union $M=\bigcup_{n\in\w}M_n$ of compact nowhere dense sets $M_n\subset M_{n+1}$ in $X$. Fix a strictly decreasing sequence $(y_n)_{n\in\w}$ of real numbers such that $\inf_{n\in\w}y_n=\inf Y$.

The equality $X'=\overline{M}^Y$ implies the equality $\dCF'(X,Y)=\bigcap_{n\in\w}\lceil M_n;y_n\rceil$, which means that $\dCF'(X,Y)$ is a $G_\delta$-set in $\dCF(X,Y)$.
\end{proof}

\begin{lemma}\label{l:F=k} The Fell hypograph topology on $\dCF'(X,Y)$ coincides with the compact-open topology.
\end{lemma}

\begin{proof} Since the Fell hypograph topology is weaker than the compact open topology, it suffices to show that each subbasic set $W$ of the compact-open topology of $\dCF'(X,Y)$ is contained in the Fell hypograph topology of the space $\dCF'(X,Y)$.

First assume that $W=\lfloor K;a\rfloor\cap\dCF'(X,Y):=\{f\in\dCF'(X,Y):\min(K)>a\}$ for some non-empty compact set $K\subset X$ and some $a\in\IR$. Fix any function $f\in W$ and observe that $f\in \dCF'(X,Y)$ and $\min f(K)>a$ imply that $a<\inf Y$ or $K\cap X'=\emptyset$. If $a<\inf Y$, then $W=\dCF'(X,Y)$ is (trivially) open in the Fell hypograph topology of the space $\dCF'(X,Y)$. If $K\cap X'=\emptyset$, then $K$ is finite and open in $X$. Then $f\in \bigcap_{x\in K}\lceil\{x\};a\rfloor\cap\dCF'(X,Y)\subset W$, which means that $W$ is open in the Fell hypograph topology of $\dCF'(X,Y)$.

If  $W=\lceil K;b\rceil\cap\dCF'(X,Y)$ for some non-empty compact set $K\subset X$ and some $b\in\IR$, then by definition, $W$ is open in the Fell hypograph topology on $\dCF'(X,Y)$.
\end{proof}

Lemma~\ref{l:F=k} allows us to identify the subspace $\dCF'(X,Y)$ of $\dCF(X,Y)$ with the subspace
$$C_k'(X,Y):=\big\{f\in C_k(X,Y):f(X')\subset\{\inf Y\}\big\}=[X';\{\inf Y\}]$$of the function space $C_k(X,Y)$ endowed with the compact-open topology.

The Baire category properties of the function spaces $C_k'(X,Y)$ are described in the following theorem, proved in \cite{BW}.

\begin{theorem}\label{t:BW} Let $X$ be a topological space containing an isolated point, and $Y\subset \IR$ be a set with $\inf Y\in Y\ne\{\inf Y\}$.
\begin{enumerate}
\item[\textup{(1)}] If $X$ does not have $\DMOP$, then the function space $C_k'(X,Y)$ is $\infty$-meager.
\item[\textup{(2)}] If $X$ has $\DMOP$ and the space $Y$ is almost Polish, then $C_k'(X,Y)$ is Baire.
\item[\textup{(3)}] $C_k'(X,Y)$ is Choquet if and only if $Y$ is almost Polish and $X$ has $\WDMOP$.
\item[\textup{(4)}] $C_k'(X,Y)$ is (almost) complete-metrizable if and only if $Y$ is (almost) Polish and $X$ is a $\dot\kappa_\w$-space.
\item[\textup{(5)}] The function space $C_k'(X,Y)$ is (almost) Polish if and only if  $Y$ is (almost) Polish and $X$ is a $\dot\kappa_\w$-space with countable set $\dot X$ of isolated points.
\item[\textup{(6)}] The function space $C_k'(X,Y)$ is separable if $Y$ is separable and $\dot X$ is countable.
\end{enumerate}
\end{theorem}

In \cite{BW} we also proved the following dichotomy for analytic spaces $C_k'(X,Y)$.

\begin{theorem}\label{t:dycha} Let $Y\subset \IR$ be a Polish subspace with $\inf Y\in Y$.
If for a topological space $X$ the function space $C_k'(X,Y)$ is analytic, then $C_k'(X,Y)$ is either Polish or $\infty$-meager.
\end{theorem}

In fact, under some assumptions, the analyticity of the function space $C'_k(X,Y)$ is equivalent to the analyticity of the function space $C'_p(X,Y)=\big\{f\in C_p(X,Y):f(X')\subset\{\inf Y\}\big\}\subset Y^X$, endowed with the topology of pointwise convergence.
The following characterization was proved in \cite{BW}.

\begin{proposition} For any non-empty subspace $Y\subset \IR$ and  a topological space $X$, the function space $C_k'(X,Y)$ is analytic if and only if $C_k'(X,Y)$ has a countable network and the function space $C_p'(X,Y)$ is analytic.
\end{proposition}

%This characterization was derived in \cite{BW} from the following characterization of analyticity of regular spaces.

%\begin{proposition} A regular space $X$ is analytic if and only if $X$ has a countable network and admits a surjective Borel map $f:A\to X$ defined on an analytic space.
%\end{proposition}

%We recall that a function $f:X\to Y$ is {\em Borel} if for any open set $U\subset Y$ the preimage $f^{-1}(U)$ is Borel in $X$.

\section{Recognizing $\infty$-meager function spaces $\dCF(X,Y)$}

In this section we find some conditions on spaces $Y\subset\IR$ and $X$ under which the function space $\dCF(X,Y)$ is $\infty$-meager.

\begin{lemma}\label{l:MY0X2} Let $Y\subset \IR$ be a non-empty subset with $\inf Y\notin Y$.
For any non-discrete $Y$-separated topological space $X$, the function space $\dCF(X,Y)$ is $\infty$-meager.
\end{lemma}

\begin{proof} Since $\inf Y\notin Y$, we can fix a strictly decreasing sequence $\{y_n\}_{n\in\w}\subset Y$ such that $\inf_{n\in\w}y_n=\inf Y$.
Take any non-isolated point $x\in X$. By Lemma~\ref{l:d}, for every $n\in\w$ the  basic open set $\lceil\{x\};y_n\rceil$ is $\infty$-dense in $\dCF(X,Y)$. Then its complement $\dCF(X,Y)\setminus\lceil\{x\};y_n\rceil$ is a closed $\infty$-codense set in $\dCF(X,Y)$. Since $$\dCF(X,Y)=\bigcup_{n\in\w}(\dCF(X,Y)\setminus\lceil\{x\};y_n\rceil),$$the space $\dCF(X,Y)$ is $\infty$-meager.
\end{proof}

A topological space $X$ is called a {\em $T_1$-space} if for each point $x\in X$ the singleton $\{x\}$ is closed in $X$. A point $x$ of a $T_1$-space is isolated if and only if the singleton $\{x\}$ is open in $X$ if and only if $\{x\}$ is not nowhere dense in $X$.

\begin{lemma}\label{l:MY0X1} Let $Y\subset \IR$ be a non-empty subset with $\inf Y\notin Y$.
For any non-discrete $T_1$-space $X$ with dense set $\dot X$ of isolated points, the function space $\dCF(X,Y)$ is $\infty$-meager.
\end{lemma}

\begin{proof} Fix a strictly decreasing sequence of real numbers $\{y_n\}_{n\in\w}\subset Y$ such that $\inf_{n\in\w}y_n=\inf Y$.
Take any non-isolated point $x\in X$. By Lemma~\ref{l:dXy}, for every $n\in\w$ the  basic open set $\lceil\{x\};y_n\rceil$ is $\infty$-dense in $\dCF(X,Y)$. Then its complement $\dCF(X,Y)\setminus\lceil\{x\};y_n\rceil$ is a closed $\infty$-codense set in $\dCF(X,Y)$, and the space $\dCF(X,Y)=\bigcup_{n\in\w}(\dCF(X,Y)\setminus\lceil\{x\};y_n\rceil)$ is $\infty$-meager.
\end{proof}

\begin{lemma}\label{l:om} Let $Y\subset \IR$ be a subset with $\inf Y\in Y\ne\{\inf Y\}$. Let $X$ be a $Y$-separated topological space containing a meager $\sigma$-compact set $M$ such that $\emptyset\ne X'^\circ\subset \overline{M}^Y$. Then the basic open set $\lceil X'^\circ,\inf Y\rfloor$ is $\infty$-meager in $\dCF(X,Y)$.
\end{lemma}

\begin{proof} Fix a strictly decreasing sequence of real numbers $(y_n)_{n\in\w}$ with  $\lim_{n\to\infty}y_n=\inf Y$. Write the meager $\sigma$-compact set $M$ as the union $M=\bigcup_{n\in\w}M_n$ of compact nowhere dense sets $M_n\subset M_{n+1}$ in $X$. By Lemma~\ref{l:d}, for every $n\in\w$ the basic open  set $\lceil M_n,y_n\rceil$ is $\infty$-dense in $\dCF(X,Y)$. Then the complement $\dCF(X,Y)\setminus\lceil M_n,y_n\rceil$ is a closed $\infty$-codense set in $\dCF(X,Y)$. It follows from $X'^\circ\subset\overline{M}^Y$ that
$$\lceil X'^\circ,\inf Y\rfloor\subset\bigcup_{n\in\w}(\dCF(X,Y)\setminus\lceil M_n,y_n\rceil),$$which means that the set $\lceil X'^\circ,\inf Y\rfloor$ is $\infty$-meager in $\dCF(X,Y)$.
\end{proof}

\begin{lemma}\label{l:MX23} Let $Y\subset \IR$ be a subset with $\inf Y\in Y\ne\{\inf Y\}$. Let $X$ be a $Y$-separated topological space containing a meager $\sigma$-compact set $M$ such that $\emptyset\ne X'^\circ\subset\overline{M}^Y$. If\/ $\inf Y\notin\dot Y$ or $\overline{X'^\circ}$ is not compact, then the function space $\dCF(X,Y)$ is $\infty$-meager.
\end{lemma}

\begin{proof} By Lemma~\ref{l:om}, the basic open set $\lceil X'^\circ,\inf Y\rfloor$ is $\infty$-meager in $\dCF(X,Y)$ and by Lemmas~\ref{l:Dni} and \ref{l:Dnc}, the set $\lceil X'^\circ,\inf Y\rfloor$ is $\infty$-dense in $\dCF(X,Y)$. Then its complement $\dCF(X,Y)\setminus\lceil X'^\circ,\inf Y\rfloor$ is closed and $\infty$-codense in $\dCF(X,Y)$. Consequently, the set
$$\dCF(X,Y)=(\dCF(X,Y)\setminus\lceil X'^\circ,\inf Y\rfloor)\cup \lceil X'^\circ,\inf Y\rfloor$$is $\infty$-meager (being the countable union of two $\infty$-meager sets in $\dCF(X,Y)$).
\end{proof}

\begin{lemma}\label{l:MXM1} Let $Y\subset \IR$ be a subset containing more than one point and $X$ be a $Y$-separable $T_1$-space with dense set $\dot X$ of isolated points. If the space $X$ does not have $\DMOP$, then the function space $\dCF(X,Y)$ is $\infty$-meager.
\end{lemma}

\begin{proof} Since the space $X$ does not have $\DMOP$, it is not discrete. If $\inf Y\notin Y$, then the space $\dCF(X,Y)$ is $\infty$-meager according to Lemma~\ref{l:MY0X1}. So, we assume that $\inf Y\in Y$. Choose a strictly decreasing sequence of real numbers $(y_n)_{n\in\w}$ with $\inf_{n\in\w}y_n=\inf Y$.

Being $Y$-separable, the space $X$ contains a meager $\sigma$-compact subset $M$ with $X'\subset\overline{M}^Y$. Write $M$ as the union $M=\bigcup_{n\in\w}M_n$ of compact nowhere dense sets $M_n\subset M_{n+1}$ in $X$.

 By Lemma~\ref{l:dXy}, for every $n\in\w$ the open basic set $\lceil M_n,y_n\rceil$ is $\infty$-dense in $\dCF(X,Y)$. Then the complement $\dCF(X,Y)\setminus\lceil M_n,y_n\rceil$ is a closed $\infty$-codense set in $\dCF(X,Y)$.

By Lemma~\ref{l:C'd}, the subspace $\dCF'(X,Y)$ is $\infty$-dense in $\dCF(X,Y)$ and by Lemma~\ref{l:F=k} and Theorem~\ref{t:BW}(1), the space $C'_k(X,Y)=\dCF'(X,Y)$ is $\infty$-meager. So, $\dCF'(X,Y)$ can be written as the countable union $\dCF'(X,Y)=\bigcup_{n\in\w}F_n$ of closed $\infty$-codense sets in $\dCF'(X,Y)$. For every $n\in\w$ let $\bar F_n$ be the closures of the set $F_n$ in $\dCF(X,Y)$.

We claim that the set $\bar  F_n$ is $\infty$-codense in $\dCF(X,Y)$. Given any compact Hausdorff space $K$ and a non-empty open set $W\subset C_k(K,\dCF(X,Y))$, we need to find a map $\mu\in W$ with $\mu(K)\cap\bar F_n=\emptyset$. Since the space $\dCF'(X,Y)$ is $\infty$-dense in $\dCF(X,Y)$, the intersection $W\cap C_k(K,\dCF'(X,Y))$ is a non-empty open set in the function space $C_k(K,\dCF'(X,Y))$. Since the set $F_n$ is $\infty$-codense in $\dCF'(X,Y)$, there exists a map $\mu\in W\cap C_k(K,\dCF'(X,Y))$ such that $\mu(K)\cap F_n=\emptyset$. Then $$\mu(K)\cap\bar F_n=(\mu(K)\cap \dCF'(X,Y))\cap\bar F_n=\mu(K)\cap (\dCF'(X,Y)\cap\bar F_n)=\mu(K)\cap F_n=\emptyset.$$

Now we see that the space $$\dCF(X,Y)=\dCF'(X,Y)\cup(\dCF(X,Y)\setminus\dCF'(X,Y))\subset \bigcup_{n\in\w}\bar F_n\cup\bigcup_{n\in\w}(\dCF(X,Y)\setminus \lceil M_n;y_n\rceil)$$is $\infty$-meager, being the countable union of
closed $\infty$-codense sets.
\end{proof}

\begin{lemma}\label{l:MXM2} Let $Y\subset \IR$ be a set containing more than one point and $X$ be a $Y$-separable $Y$-separated topological space. If the space $X$ does not have $\DMOP$, then the function space $\dCF(X,Y)$ is  $\infty$-meager.
\end{lemma}

\begin{proof} Since the space $X$ does not have $\DMOP$, it is not discrete. If $\inf Y\notin Y$, then the function space $\dCF(X,Y)$ is $\infty$-meager by Lemma~\ref{l:MY0X2}. So, we assume that $\inf Y\in Y$.

If the set $X'$ is nowhere dense in $X$, then the space $\dCF(X,Y)$ is $\infty$-meager by Lemma~\ref{l:MXM1}. So, we assume that the interior $X'^\circ$ is not empty. If the closure $\overline{X'^\circ}$ is not compact or $\inf Y\notin \dot Y$, then the space $\dCF(X,Y)$ is $\infty$-meager by Lemma~\ref{l:MX23}. So, we assume that  $\overline{X'^\circ}$ is compact and $\inf Y\in\dot Y$.

In this case we can choose a real number $\e$ such that $\{\inf Y\}=\{y\in Y:y<\e\}$ and conclude that the set $$[X'^\circ;\{\inf Y\}]=\lceil \overline{X'^\circ},\e\rceil=\dCF(X,Y)\setminus\lceil X'^\circ,\inf Y\rfloor$$  is clopen in $\dCF(X,Y)$. By Lemma~\ref{l:C'd}, the space $\dCF'(X,Y)$ is $\infty$-dense in  the clopen subspace $[{X'^\circ};\{\inf Y\}]$ of $\dCF(X,Y)$.

By Theorem~\ref{t:BW}(1), the space $\dCF'(X,Y)$ is $\infty$-meager. So, $\dCF'(X,Y)=\bigcup_{n\in\w}F_n$ for some closed $\infty$-codense sets $F_n$ in $\dCF'(X,Y)$.
Let $\bar F_n$ be the closure of the set $F_n$ in the space $\dCF(X,Y)$.
Taking into account that  $\dCF'(X,Y)$ is $\infty$-dense in $[{X'^\circ};\{\inf Y\}]$, we can show that each set $\bar F_n$ is $\infty$-codense in $[{X'^\circ};\{\inf Y\}]$ by analogy with the proof of Lemma~\ref{l:MXM1}. Since  $[{X'^\circ};\{\inf Y\}]$ is clopen in $\dCF(X,Y)$, the set $\bar F_n$ is $\infty$-codense in $\dCF(X,Y)$.

Fix a strictly decreasing sequence of real numbers $(y_n)_{n\in\w}$ with $\lim_{n\to\infty}y_n=\inf Y$.  Being $Y$-separable, the space $X$ contains a meager $\sigma$-compact $M$ such that $X'\subset\overline{M}^Y$. Write $M$ as the union $M=\bigcup_{n\in\w}M_n$ of compact nowhere dense sets $M_n\subset M_{n+1}$ in $X$. By Lemma~\ref{l:MY0X2}, for every $n\in\w$ the open basic set $\lceil M_n,y_n\rceil$ is $\infty$-dense in $\dCF(X,Y)$. Then the complement $\dCF(X,Y)\setminus\lceil M_n,y_n\rceil$ is a closed $\infty$-codense set in $\dCF(X,Y)$.

Now we see that the space
$$\dCF(X,Y)=\dCF'(X,Y)\cup(\dCF(X,Y)\setminus\dCF'(X,Y))\subset\bigcup_{n\in\w}\bar F_n\cup\bigcup_{n\in\w}(\dCF(X,Y)\setminus \lceil M_n,y_n\rceil)$$is $\infty$-meager (being the countable union of closed $\infty$-codense sets).
\end{proof}

\begin{lemma}\label{l:m} Let $Y\subset \IR$ be a subset, $Z\subset Y$ be an open zero-dimensional subspace in $Y$ and $F\subset Z$ be a closed nowhere dense subset in $Y$. Then the set $F$ is $\infty$-codense in $Y$.
\end{lemma}

\begin{proof} Given any compact Hausdorff space $K$, a continuous map $f:K\to Y$ and a neighborhood $O_f\subset C_k(K,Y)$, we need to find a continuous map $g\in O_f$ such that $g(K)\cap F=\emptyset$.  By the normality of the space $Y$, there exists an open set $W\subset Y$ such that $F\subset W\subset\overline{W}\subset Z$. The zero-dimensional space $Z$, being metrizable and separable, has large inductive dimension zero, see \cite[7.3.3]{Eng}. Consequently,  there exists a clopen set $V\subset Z$ such that $F\subset V\subset W$. Since $V\subset W\subset \overline{W}\subset Z$, the clopen subset $V$ of $Z$ remains clopen in the space $Y$.

By \cite[8.2.7]{Eng}, the compact-open topology on the function space $C_k(K,Y)$ is generated by the metric $\rho(g,g')=\sup_{x\in K}|g(x)-g'(x)|$, where $g,g'\in C_k(K,Y)$. Consequently, we can find $\e>0$ such that any map $g:K\to Y$ with $\rho(f,g)<\e$ belongs to the neighborhood $O_f$ of $f$. The space $f(K)\cap V$ is compact and zero-dimensional. So, admits a finite cover $\{V_1,\dots,V_n\}$ by pairwise disjoint clopen sets in $V$ of diameter $<\e$. For each $i\le n$ choose a point $v_i\in V_i\setminus F$. Consider the continuous map $g:K\to Y$ defined by the formula:
$$g(x)=\begin{cases}v_i&\mbox{if $f(x)\in V_i$ for some $i\le n$;}\\
f(x)&\mbox{otherwise.}
\end{cases}
$$
Then $\rho(f,g)<\e$ and hence $g\in O_f$. Also $$g(K)=g(f^{-1}(f(K)\cap V))\cup g(f^{-1}(f(K)\setminus V))\subset\{v_1,\dots,v_n\}\cup (f(K)\setminus V)\subset Y\setminus F.$$
\end{proof}

\begin{lemma}\label{l:MYMx} For any meager space $Y\subset\IR$ and any topological space $X$ containing an isolated point $x$, the function space $\dCF(X,Y)$ is $\infty$-meager.
\end{lemma}

\begin{proof}  Assume that the space $Y\subset\IR$ is meager. Then $Y$ can be written as the countable union $Y=\bigcup_{n\in\w}Y_n$ of closed nowhere dense sets $Y_n$. Being a meager subset of the real line, the space $Y$ is zero-dimensional. By Lemma~\ref{l:m}, the set $Y_n$ is $\infty$-codense in $Y$.
Since the point $x$ of $X$ is isolated,  the map
$$H:\dCF(X,Y)\to Y\times \dCF(X\setminus\{x\},Y),\;\;H:f\mapsto (f(x),f{\restriction}X\setminus\{x\}),$$is a homeomorphism.
The $\infty$-codensity of the closed set $Y_n$ in $Y$ implies the $\infty$-codensity of the closed set $Y_n\times \dCF(X\setminus\{x\},Y)$ in $Y\times \dCF(X\setminus \{x\},Y)=H(\dCF(X,Y))$. Since $H$ is a homeomorphism, the closed set
 $$F_n=\{f\in \dCF(X,Y):f(x)\in Y_n\}=H^{-1}\big(Y_n\times \dCF(X\setminus\{x\},Y)\big)$$ is $\infty$-codense in $\dCF(X,Y)$. Then the space $\dCF(X,Y)=\bigcup_{n\in\w}F_n$ is $\infty$-meager, being a countable union of closed $\infty$-codense sets $F_n$, $n\in\w$.
\end{proof}

\begin{lemma}\label{l:MYM} Let $Y\subset\IR$ be a subspace containing more than one point and $X$ be a non-empty  $Y$-separable $Y$-separated space. If the space $Y$ is meager, then the function space $\dCF(X,Y)$ is $\infty$-meager.
\end{lemma}

\begin{proof} If the space $X$ contains an isolated point, then the function space $\dCF(X,Y)$ is $\infty$-meager by Lemma~\ref{l:MYMx}. So, we assume that the space $X$ contains no isolated points. If $\inf Y\notin Y$, then the space $\dCF(X,Y)$ is $\infty$-meager by Lemma~\ref{l:MY0X2}. So, we assume that $\inf Y\in Y$. Since the space $Y$ is meager, the point $\inf Y$ is not isolated in $Y$. In this case the space $\dCF(X,Y)$ is $\infty$-meager by Lemma~\ref{l:MX23}.
\end{proof}

\begin{lemma}\label{l:MYNd} Let $Y\subset \IR$ be a non-empty subspace, $X$ be a topological space, and $D\subset\dot X$ be an infinite closed set in $X$.
If the space $Y$ is not Baire, then the function space $\dCF(X,Y)$ is $\infty$-meager.
\end{lemma}

\begin{proof} Replacing $D$ by a smaller infinite subset, we can assume that $D$ is countable.

The space $Y$ is not Baire and hence contains a non-empty open meager subset $Z\subset Y$. Being a meager subset of the real line, the space $Z$ is zero-dimensional. Being a meager $F_\sigma$-set in $Y$, the space $Z$ can be written as  the countable union $Z=\bigcup_{n\in\w}Z_n$ of nowhere dense closed subsets $Z_n$ of $Y$.

Observe that $$\dCF(X,Y)=\{f\in \dCF(X,Y):f(D)\subset Y\setminus Z\}\cup\bigcup_{x\in D}\bigcup_{n\in\w}\{f\in \dCF(X,Y):f(x)\in Z_n\}.$$
Using Lemma~\ref{l:m}, it can be shown that for every $x\in D$ and $n\in\w$ the closed set $\{f\in \dCF(X,Y):f(x)\in Z_n\}$ is $\infty$-codense in $\dCF(X,Y)$. It remains to prove that the closed set $[D;Y{\setminus}Z]:=\{f\in \dCF(X,Y):f(D)\subset Y\setminus Z\}$ is $\infty$-codense in $\dCF(X,Y)$.

Given a continuous map $\mu:K\to \dCF(X,Y)$, defined on a compact Hausdorff space $K$ and a neighborhood $O_\mu\subset C_k(K,\dCF(X,Y))$ of $\mu$, we need to find a continuous map $\mu'\in O_\mu$ such that $\mu'(K)\subset \dCF(X,Y)\setminus [D;Y\setminus Z]$. We can assume that the neighborhood $O_\mu$ is of basic form $O_\mu=\bigcap_{i=1}^m[K_i;B_i]$ where $K_1,\dots,K_m$ are compact sets in $K$ and each set $B_i\subset\dCF(X,Y)$ is of basic form $B_i=\bigcap_{U\in\U_i}\lceil U;a_i(U)\rfloor\cap\bigcap_{\kappa\in\K_i}\lceil \kappa,b_i(\kappa)\rceil,$ where $\U_i$ is a non-empty finite family of non-empty open sets in $X$, $\K_i$ is a non-empty finite family of non-empty compact sets in $X$ and $a_i:\U_i\to\IR$, $b_i:\K_i\to \IR$ are functions. We can also assume that for every $i\le m$ and $U\in\U_i$ the open set is either singleton $\{x\}\subset\dot X$ or $U\subset X'^\circ$. In this case the set $E=\bigcup_{i\le m} (\dot X\cap\bigcup\U_i)$ is finite.

Since the infinite set $D$ is discrete and closed in $X$, there exists a point $d\in D\setminus(E\cup\bigcup_{i=1}^n\bigcup\K_i)$. Fix any point $z\in Z$ and consider the continuous map $s:\dCF(X,Y)\to \dCF(X,Y)\setminus [D,Y\setminus Z]$ assigning to each function $f\in\dCF(X,Y)$ the function $f':X\to Y$ such that $f'(d)=z$ and $f'{\restriction}X\setminus\{d\}=f{\restriction}X\setminus \{d\}$. Then the continuous map $\mu'=s\circ\mu:K\to \dCF(X,Y)\setminus [D;Y\setminus Z]$ belongs to the neighborhood $O_\mu$, witnessing that the closed set $[D;Y\setminus Z]$ is $\infty$-codense in $\dCF(X,Y)$.
\end{proof}

\begin{lemma}\label{l:MYN} Let $Y\subset \IR$ be a non-empty subspace and $X$ be a $Y$-separable $Y$-separated space  such that the set $\dot X$ is not contained in a compact subset of $X$.
If the space $Y$ is not Baire, then the function space $\dCF(X,Y)$ is $\infty$-meager.
\end{lemma}

\begin{proof} If the topological space $X$ does not have $\DMOP$, then the function space $\dCF(X,Y)$ is $\infty$-meager by Lemma~\ref{l:MXM2}. So, we assume that the space $X$ has $\DMOP$. Since the set $\dot X$ is not contained in a compact subset of $X$, the family of singletons $\{\{x\}:x\in\dot X\}$ is moving off. Since $X$ has $\DMOP$, this family has an infinite discrete subfamily, which implies that $X$ contains an infinite subset $D\subset\dot X$, which is closed in $X$.
Now we can apply Lemma~\ref{l:MYNd} to conclude that the space $\dCF(X,Y)$ is $\infty$-meager.
\end{proof}

\section{Recognizing Baire spaces among function spaces $\dCF(X,Y)$}\label{s:B2}

\begin{lemma}\label{l:BX0} Let $Y$ be a Baire subspace of the real line. For any discrete topological space $X$, the function space $\dCF(X,Y)$ is Baire.
\end{lemma}

\begin{proof} Taking into account that $Y$ is second countable and applying \cite[Theorem 3]{Oxtoby}, we conclude that the Tychonoff power $Y^X$ is Baire. Since $X$ is discrete, the Fell hypograph topology on $C(Y,X)=Y^X$ coincides with the Tychonoff product topology on $Y^X$, which implies that the function space $\dCF(X,Y)$ is Baire.
\end{proof}

\begin{lemma} Let $Y\subset\IR$ be a non-empty space with $\inf Y\notin Y$ and $X$ be a $T_1$-space with dense set $\dot X$ of isolated points. If the function space $\dCF(X,Y)$ is Baire, then $X$ is discrete and $Y$ is Baire.
\end{lemma}

\begin{proof} Being Baire, the space $\dCF(X,Y)$ is not meager and by Lemma~\ref{l:MY0X1}, the space $X$ is discrete. In this case the Fell hypograph topology on $C(X,Y)$ coincides with the topology of pointwise convergence on $C(X,Y)=Y^X$, which implies that the Tychonoff power $Y^X$ is Baire and so is the space $Y$.
\end{proof}

\begin{lemma}\label{l:B1} Let $Y\subset\IR$ be a non-empty space with $\inf Y\notin Y$ and $X$ be a $Y$-separated space. If the function space $\dCF(X,Y)$ is Baire, then $X$ is discrete and $Y$ is Baire.
\end{lemma}

\begin{proof} Being Baire, the space $\dCF(X,Y)$ is not meager and by Lemma~\ref{l:MY0X2}, the space $X$ is discrete. In this case the Fell hypograph topology on $C(X,Y)$ coincides with the topology of pointwise convergence on $C(X,Y)=Y^X$, which implies that the Tychonoff power $Y^X$ is Baire and so is the space $Y$.
\end{proof}

These three lemmas imply the following characterization.

\begin{lemma} \label{l:Bnon} Let $X$ be a non-empty $T_1$-space and $Y\subset\IR$ be a non-empty space with $\inf Y\notin Y$. Assume that $X$ is $Y$-separated or $\dot X$ is dense in $X$. The function space $\dCF(X,Y)$ is Baire if and only if $X$ is discrete and $Y$ is Baire.
\end{lemma}

\begin{lemma}\label{l:d'} Let $Y\subset \IR$ be a subspace with $\inf Y\in Y\ne\{\inf Y\}$, and $X$ be a $Y$-separable $Y$-separated space. If the function space $\dCF(X,Y)$ is Baire, then the set $\dot X$ of isolated points is dense in $X$.
\end{lemma}

\begin{proof} To derive a contradiction, assume that the set $X'$ of isolated points of $X$ has non-empty interior $X'^\circ$. Being $Y$-separable, the space $X$ contains meager $\sigma$-compact subset $M$ with $X'\subset\overline{M}^Y$. Write $M$  as the countable union $M=\bigcup_{n\in\w}M_n$ of nowhere dense compact sets $M_n\subset M_{n+1}$ in $X$.

Choose a strictly decreasing sequence of real numbers $(y_n)_{n\in\w}$ such that $\lim_{n\to\infty}y_n=\inf Y$. By Lemma~\ref{l:d}, for any $n\in\w$ the basic open set $\lceil M_n;y_n\rceil$ is dense in $\dCF(X,Y)$ and hence its complement $\dCF(X,Y)\setminus\lceil M_n;y_n\rceil$ is closed and nowhere dense in $\dCF(X,Y)$.

Since $\lceil X'^\circ;\inf Y\rfloor\subset \bigcup_{n\in\w}(\dCF(X,Y)\setminus\lceil M_n;y_n\rceil)$, the non-empty basic open set $\lceil X'^\circ;\inf Y\rfloor$ is meager. So, $\dCF(X,Y)$ cannot be Baire.
\end{proof}

\begin{lemma}\label{l:B1} Let $Y\subset \IR$ be a subset with $\inf Y\in Y$ and $X$ be a $T_1$-space with dense set $\dot X$ of isolated points.
\begin{enumerate}
\item[\textup{(1)}] If $X$ has $\DMOP$ and $Y$ is almost Polish, then $\dCF(X,Y)$ is Baire;
\item[\textup{(2)}] If $\dCF(X,Y)$ is not meager and $X$ is $Y$-separable, then the space $X$ has $\DMOP$;
\item[\textup{(3)}] If $\dCF(X,Y)$ is Baire, then so is the space $Y$.
\end{enumerate}
\end{lemma}

\begin{proof} 1. Assume that the space $X$ has $\DMOP$ and the space $Y$ is almost Polish. %Then  $Y$ contains a dense Polish subspace $P$. The density of the space $P$ in $Y$ implies the density of the space $\dCF'(X,P)$ in $\dCF'(X,Y)$.
By Lemma~\ref{l:C'd}, the space $\dCF'(X,Y)$ is dense in $\dCF(X,Y)$.
%, which implies that $\dCF'(X,P)$ is dense in $\dCF(X,Y)$. By \cite[3.11]{Ke}, the Polish space $P$ is a $G_\delta$-set in the real line.
By Lemma~\ref{l:F=k} and Theorem~\ref{t:BW}(2), the function space $C_k'(X,Y)=\dCF'(X,Y)$ is Baire. Then the space $\dCF(X,Y)$ is Baire, too (because it contains a dense Baire subspace).
\smallskip

2. Assume that the function space $\dCF(X,Y)$ is not meager and the space $X$ is $Y$-separable. By Lemmas~\ref{l:C'd} and \ref{l:Gdelta}, $\dCF'(X,Y)$ is a dense $G_\delta$-set in $\dCF(X,Y)$, which implies that the complement $\dCF(X,Y)\setminus \dCF'(X,Y)$  is meager in $\dCF(X,Y)$. Assuming that the space $\dCF'(X,Y)$ is meager, we would conclude that the space $\dCF(X,Y)=\dCF'(X,Y)\cup(\dCF(X,Y)\setminus\dCF'(X,Y))$ is meager, which contradicts our assumption. This contradiction shows that the space $\dCF'(X,Y)$ is not meager. By Lemma~\ref{l:F=k},
$\dCF'(X,Y)=C_k'(X,Y)$ and by Theorem~\ref{t:BW}(1), the space $X$ has $\DMOP$.
\smallskip

3. Assuming that the space $\dCF(X,Y)$ is Baire, we shall prove that the space $Y$ is Baire. Take any isolated point $x\in X$ and consider the subspace $Z:=X\setminus\{x\}$ of $X$. It is easy to see that the map $$H:\dCF(X,Y)\to \dCF(Z,Y)\times Y,\;\;H:f\mapsto (f{\restriction}Z,f(x)),$$is a homeomorphism. Then the product $\dCF(Z,Y)\times Y$ is Baire and so is the space $Y$.
\end{proof}

\begin{lemma}\label{l5.2} Let $Y\subset \IR$ be an Polish+meager space with $\inf Y\in Y\ne\{\inf Y\}$. For any $Y$-separable space $X$ with dense set $\dot X$ of  isolated points, the following conditions are equivalent:
\begin{enumerate}
\item[\textup{(1)}] the function space  $\dCF(X,Y)$ is Baire;
\item[\textup{(2)}] $Y$ is Baire and $X$ has $\DMOP$.
\end{enumerate}
\end{lemma}

\begin{proof} $(1)\Ra(2)$ If the function space $\dCF(X,Y)$ is Baire, then by Lemma~\ref{l:B1}(2,3), the space $X$ has $\DMOP$ and the space $Y$ is Baire.
\smallskip

$(2)\Ra(1)$ Assume that the space $Y$ is Baire and the space $X$ has $\DMOP$.
By definition, the Polish+meager space $Y$ contains a Polish subspace $P\subset Y$ whose complement $Y\setminus P$ is meager in $Y$. We claim that the Polish space $P$ is dense in $Y$. In the opposite case the non-empty open subset $Y\setminus\bar P$ of $Y$ is meager and $Y$ cannot be Baire. Now Lemma~\ref{l:B1}(1) implies that the function space $\dCF(X,Y)$ is Baire.
\end{proof}

Combining Lemmas~~\ref{l:d'} and \ref{l5.2}, we obtain the following proposition which implies Theorem~\ref{t:main}(1) announced in the introduction.

\begin{proposition}\label{p:Bm2} Let $Y\subset \IR$ be a Polish+meager subspace such that $\inf Y\in Y\ne\{\inf Y\}$. For any $Y$-separable $Y$-separated space, the following conditions are equivalent:
\begin{enumerate}
\item[\textup{(1)}] the function space $\dCF(X,Y)$ is Baire;
\item[\textup{(2)}] the space $Y$ is Baire, the space $X$ has $\DMOP$ and the set $\dot X$ is dense in $X$.
\end{enumerate}
\end{proposition}

\section{Recognizing Choquet spaces among function spaces $\dCF(X,Y)$}

We shall use the following known properties of Choquet spaces, see \cite[8.13]{Ke} and  \cite{White}.

\begin{lemma}\label{l:bC}
\begin{enumerate}
\item[\textup{(1)}] The Tychonoff product of Choquet spaces is Choquet.
\item[\textup{(2)}] Each dense $G_\delta$-set of a Choquet space $X$ is Choquet.
\item[\textup{(3)}] A topological space is Choquet if it contains a dense Choquet subspace.
\item[\textup{(4)}] An open continuous image of a Choquet space is Choquet.
\item[\textup{(5)}] A metrizable space is Choquet if and only if it is almost complete-metrizable.
\item[\textup{(6)}] A metrizable separable space is Choquet if and only if it is almost Polish.
\end{enumerate}
\end{lemma}

%The following three lemmas are proved in \cite{BW}.

%\begin{lemma} Let $Y\subset\IR$ be a non-empty space with $\inf Y\in Y\ne\{\inf Y\}$. For a topological space $X$ the function space $C'_k(X,Y)$ is Choquet if and only if $Y$ is Choquet and $X$ has $\WDMOP$.
%\end{lemma}

%\begin{lemma} A topological space $X$ with countable set $\dot X$ of isolated points has $\WDMOP$ if and only if it is a $\dot\kappa_\w$-space.
%\end{lemma}

%\begin{lemma} Let $X$ be a topological space with $\WDMOP$. A countable subset $A\subset\dot X$ is closed in $X$ if and only if for any compact subset $K\subset X$ the intersection $K\cap A$ is finite.
%\end{lemma}

\begin{lemma}\label{l:C1} Let $Y$ be a non-empty almost Polish subspace of the real line. Then for any discrete topological space $X$ the function space $\dCF(X,Y)$ is Choquet.
\end{lemma}

\begin{proof}
Since $X$ is discrete, the Fell hypograph topology on $\dCF(X,Y)$ coincides with the topology of pointwise convergence. So, $\dCF(X,Y)$ can be identified with the Tychonoff power $Y^X$, which is Choquet by Lemma~\ref{l:bC}(1,6).
\end{proof}

\begin{lemma}\label{l:C2} Let $Y\subset\IR$ be a non-empty space and $X$ be a topological space containing an isolated point $x$. If the function space $\dCF(X,Y)$ is Choquet, then $Y$ is Choquet and almost Polish.
\end{lemma}

\begin{proof} Since the point $x$ is isolated in $X$, the map $\delta_x:\dCF(X,Y)\to Y$, $\delta_x:f\mapsto f(x)$, is surjective, continuous and open. If $\dCF(X,Y)$ is  Choquet, then its open continuous image $Y$ is Choquet and almost Polish by Lemma~\ref{l:bC}(4,6).
\end{proof}

\begin{lemma}\label{l:C3} Let $X$ be a non-empty $T_1$-space and $Y\subset \IR$ be a non-empty subspace with $\inf Y\notin Y$. Assume that $X$ is $Y$-separated or $\dot X$ is dense in $X$. The function space $\dCF(X,Y)$ is Choquet if and only if $X$ is discrete and $Y$ is almost Polish.
\end{lemma}

\begin{proof} The ``if'' part follows from Lemma~\ref{l:C1}. To prove the ``only if'' part, assume that the function space $\dCF(X,Y)$ is Choquet. Then it is Baire and by Lemma~\ref{l:Bnon}, the space $X$ is discrete and hence has an isolated point. By Lemma~\ref{l:C2}, the space $Y$ is almost Polish.
\end{proof}

\begin{lemma}\label{l:C4} Let $Y\subset \IR$ be a subspace with $\inf Y\in Y\ne\{\inf Y\}$ and let $X$ be a $T_1$-space with dense set $\dot X$ of isolated points. The function space $\dCF(X,Y)$ is Choquet if the space $Y$ is almost Polish and the space $X$ has $\WDMOP$.
\end{lemma}

\begin{proof} By Lemma~\ref{l:C'd}, the set $\dCF'(X,Y)$ is dense in $\dCF(X,Y)$ and by Lemma~\ref{l:F=k}, $\dCF'(X,Y)$ is homeomorphic to  the function space $C'_k(X,Y)$.

If $Y$ is almost Polish and $X$ has $\WDMOP$, then by Theorem~\ref{t:BW}(3), the function space $C_k'(X,Y)$ is Choquet and so is its topological copy $\dCF'(X,Y)$. Then the space $\dCF(X,Y)$ is Choquet since it contains a dense Choquet subspace $\dCF'(X,Y)$.
 \end{proof}

\begin{lemma}\label{l:C5} Let $Y\subset \IR$ be a subspace with $\inf Y\in Y\ne\{\inf Y\}$ and let $X$ be a $Y$-separable $T_1$-space with dense set $\dot X$ of isolated points. The function space $\dCF(X,Y)$ is Choquet if and only if the space $Y$ is almost Polish and $X$ has $\WDMOP$.
\end{lemma}

\begin{proof} The ``if'' part follows from Lemma~\ref{l:C4}. To prove the ``only if'' part, assume that the function space $\dCF(X,Y)$ is Choquet. By Lemmas~\ref{l:C'd} and \ref{l:Gdelta}, $\dCF'(X,Y)$ is a dense $G_\delta$-set in $\dCF(X,Y)$. By Lemma~\ref{l:bC}(2), the space $\dCF'(X,Y)$ is Choquet and so is its topological copy $C_k'(X,Y)$. Applying  Theorem~\ref{t:BW}(3), we conclude that the space $Y$ is Choquet and the space $X$ has $\WDMOP$.
\end{proof}

The following proposition implies Theorem~\ref{t:main}(2) announced in the introduction.

\begin{proposition}\label{p:C} Let $Y\subset \IR$ be a subspace with $\inf Y\in Y\ne\{\inf Y\}$. For a $Y$-separable $Y$-separated space $X$, the function space $\dCF(X,Y)$ is Choquet if and only if  the space $Y$ is almost Polish, the set $\dot X$ is dense in $X$, and the space $X$ has $\WDMOP$.
\end{proposition}

\begin{proof} The ``if'' part is proved in Lemma~\ref{l:C5}. To prove the ``only'' if part, assume that the function space $\dCF(X,Y)$ is Choquet. Then it is Baire and by Lemma~\ref{l:d'}, the set $\dot X$ is dense in $X$. By Lemma~\ref{l:C5}, the space $Y$ is almost Polish and $X$ has $\WDMOP$.
\end{proof}

\section{Recognizing strong Choquet spaces among function spaces $\dCF(X,Y)$}

We shall use the following known properties of strong Choquet spaces, see \cite[8.16, 8.17]{Ke}.

\begin{lemma}\label{l:bsC}
\begin{enumerate}
\item[\textup{(1)}] The Tychonoff product of strong Choquet spaces is strong Choquet.
\item[\textup{(2)}] An open continuous image of a strong Choquet space is strong Choquet.
\item[\textup{(3)}] A metrizable separable space is strong Choquet if and only if it is  Polish.
\end{enumerate}
\end{lemma}

\begin{lemma}\label{l:sC1} Let $Y\subset \IR$ be a non-empty subspace. If $Y$ is Polish, then for any discrete topological space $X$ the function space $\dCF(X,Y)$ is strong Choquet.
\end{lemma}

\begin{proof}
Since $X$ is discrete, the Fell hypograph topology on $\dCF(X,Y)$ coincides with the topology of pointwise convergence. So, $\dCF(X,Y)$ can be identified with the Tychonoff power $Y^X$, which is strong Choquet by Lemma~\ref{l:bsC}(1,3).
\end{proof}

\begin{lemma}\label{l:sC2} Let $Y\subset\IR$ be a non-empty space and $X$ be a topological space containing an isolated point $x$. If the function space $\dCF(X,Y)$ is strong Choquet, then $Y$ is Polish.
\end{lemma}

\begin{proof} Since the point $x$ is isolated in $X$, the map $\delta_x:\dCF(X,Y)\to Y$, $\delta_x:f\mapsto f(x)$, is surjective, continuous and open. If $\dCF(X,Y)$ is strong Choquet, then its open continuous image $Y$ is strong Choquet and Polish by Lemma~\ref{l:bsC}(2,3).
\end{proof}

\begin{lemma}\label{l:sC3} Let $X$ be a topological space and  $Y\subset \IR$ be a non-empty subspace such that $\inf Y\notin Y$. Assume that $X$ is $Y$-separated or the set $\dot X$ is dense in $X$. The function space $\dCF(X,Y)$ is strong  Choquet if and only if $X$ is discrete and $Y$ is Polish.
\end{lemma}

\begin{proof}  The ``if'' part follows from Lemma~\ref{l:sC1}. To prove the ``only if'' part, assume that the function space $\dCF(X,Y)$ is strong Choquet. Then it is Baire and by Lemma~\ref{l:Bnon}, the space $X$ is discrete. Then $x$ has an isolated point and by Lemma~\ref{l:sC2}, the space $Y$ is Polish.
\end{proof}

The case of $Y$ with $\inf Y\in Y$ is more complicated and requires playing the strong Choquet game on the function space $\dCF(X,Y)$.

\begin{lemma}\label{l:SCg} Let $Y\subset \IR$ be a subset containing more than one point. If a topological space $X$ contains a metrizable compact subset $K\subset X$ with infinite intersection $K\cap\dot X$, then the player $\mathsf E$ has a winning strategy in the strong Choquet game $\SGEN(\dCF(X,Y))$.
\end{lemma}

\begin{proof} Since the intersection $\dot X\cap K$ contains a non-trivial convergent sequence, we can replace $K$ by a smaller compact space and assume that $K\cap\dot X$ is dense in $K$ and $K$ has a unique non-isolated point $x'\in K\cap X'$. Write the countable infinite set $K\cap \dot X$ as the union $K\cap\dot X=\bigcup_{n\in\w}F_n$ of an increasing sequence $(F_n)_{n\in\w}$ of finite sets.

By our assumption, the set $Y$ contains two real numbers $\underline{u}<\bar u$.

Let $\tau$ be the family of all non-empty open sets in $\dCF(X,Y)$ and let $$[X',\{\bar u\}]:=\{f\in \dCF(X,Y):f(X')\subset\{\bar u\}\}.$$ For every sequence $s=(W_0,\dots,W_n)\in\tau^{<\w}$  of non-empty open sets we shall define a function $f_s\in W_n$, a neighborhood $V_s\subset W_n$ of $f_s$, and two points $v_s,w_s\in\dot X\cap K\setminus F_n$ such that if $W_n\cap[X';\{\bar u\}]\ne\emptyset$, then
$$f_s\in [X',\{\bar u\}] \mbox{ \ and \ }f_s\in V_s\subset W_n\cap \lceil \{w_s\},\tfrac13\underline{u}+\tfrac23\bar u\rfloor\cap\lceil\{v_s\};\tfrac23\underline{u}+\tfrac13\bar u\rceil.$$
Now we shall explain how to construct $f_s$, $V_s$, $v_s$ and $w_s$.

If $W_n\cap[X';\{\bar u\}]=\emptyset$, then put $V_s=W_n$, $f_s$ be any element of $V_s$, and $v_s,w_s\in\dot X\cap K\setminus F_n$ be any distinct points.

If $W_n\cap[X';\{\bar u\}]\ne\emptyset$, then choose any function $g_s\in W_n\cap[X',\{\bar u\}]$. Next, using the definition of the Fell hypograph topology, we can find a finite family $\U_s$ of non-empty open sets in $X$, a non-empty finite family $\K_s$ of non-empty compact sets in $X$ and two functions $a_s:\U_s\to\IR$, $b_s:\K_s\to\IR$ such that the basic open set
$$\lceil\U_s,\K_s;a_s,b_s]:=\bigcap_{U\in\U_s}\lceil U;a_s(U)\rfloor\cap\bigcap_{\kappa\in\K_i}\lceil\kappa;b_s(\kappa)\rceil$$is a neighborhood of $g_s$, contained in $W_n$. Replacing the sets $U\in\U_s$ by smaller sets, we can assume that each set $U\in\U_s$ intersecting the set $\dot X$ is a singleton. In this case the union $\bigcup\dot\U_s$ of the family $\dot\U_s=\{U\in\U_s:U\cap\dot X\ne\emptyset\}$ is finite. Since each compact subset of $\dot X$ is finite, the union $\bigcup\dot\K_s$ of the family $\dot\K_s=\{\kappa\in\K_s:\kappa\cap X'=\emptyset\}$ also is finite.

Let $\underline{b}_s=\min\{b_s(\kappa):\kappa\in\K_s\setminus \dot\K_s\}$. We claim that $\bar u< \underline{b}_s$. Indeed, find $\kappa\in\K_s\setminus\dot \K_s$ with $\underline{b}_s=b_s(\kappa)$ and observe that $g_s\in \lceil\kappa,b_s(\kappa)\rceil$ implies that $\bar u=\max g_s(\kappa\cap X')\le\max g_s(\kappa)<b_s(\kappa)=\underline{b}_s$.

Using the continuity of the function $g_s$ at the unique accumulation point $x'$ of the compact set $K$, we can find a point $w_s\in\dot X\cap K\setminus(F_n\cup \bigcup(\dot\U\cup\dot\K))$ such that $$\tfrac13\underline{u}+\tfrac23\bar u<g_s(w_s)<\underline{b}_s.$$ Next, choose any point $v_s\in \dot X\cap K\setminus(\{w_s\}\cup F_n\cup \bigcup(\dot\U\cup\dot\K))$.
Put $$V_s:=\lceil\U_s,\K_s;a_s,b_s]\cap \lceil \{w_s\},\tfrac13\underline{u}+\tfrac23\bar u\rfloor\cap\lceil\{v_s\};\tfrac23\underline{u}+\tfrac13\bar u\rceil.$$
Finally, define a function $f_s\in\dCF(X,Y)$ letting $f_s(v_s)=\underline{u}$ and $f_s(x)=g_s(x)$ for any $x\in X\setminus\{v_s\}$. It is easy to see that $f_s,V_s,v_s,w_s$ have the required properties.

 Now we define a strategy $S_{\mathsf E}$ of the player $\mathsf E$ in the strong Choquet game $\SGEN(\dCF(X,Y))$ assigning to each $s=(W_0,\dots,W_n)\in\tau^{<\w}$ of non-empty open sets of $\dCF(X,Y)$ the pair $(f_s,V_s)$. For the empty sequence, we assume that $V_\emptyset=\dCF(X,Y)$ and $f_\emptyset:X\to\{\bar u\}$ is the constant function. We claim that this strategy of the player $\mathsf E$ is winning. Given any sequence $s=(W_n)_{n\in\w}\in\tau^\w$ with $f_{s{\restriction}n}\in W_n\subset V_{s{\restriction}n}$ for every $n\in\w$, we need to show that the intersection $\bigcap_{n\in\w}W_n=\bigcap_{n\in\w}V_{s{\restriction}n}$ is empty. To derive a contradiction, assume that this intersection contains some function $f\in\dCF(X,Y)$.

   By induction it can be shown that $f_{s{\restriction}n}\in [X',\{\bar u\}]$ and hence
 $$f\in V_{s{\restriction}n}\subset  \lceil \{w_{s{\restriction}n}\},\tfrac13\underline{u}+\tfrac23\bar u\rfloor\cap\lceil\{v_{s{\restriction}n}\};\tfrac23\underline{u}+\tfrac13\bar u\rceil,$$which contradicts the continuity of $f$ as the sequences $(v_{s{\restriction}n})_{n\in\w}$ and $(w_{s{\restriction}n})_{n\in\w}$ both accumulate at the unique non-isolated point $x'$ of the compact set $K$.
\end{proof}

The following proposition implies Theorem~\ref{t:main}(3) announced in the introduction.

\begin{proposition}\label{p:sC} Let $Y\subset \IR$ be a subspace containing more than one point and $X$ be a non-empty $Y$-separable $Y$-separated space.  The function space $\dCF(X,Y)$ is strong Choquet if (and only if) the space $Y$ is Polish, $\dot X$ is dense in $X$ and the set $\dot X$ is (sequentially) closed in $X$.
\end{proposition}

\begin{proof} To prove the ``if'' part, assume that $Y$ is Polish, $\dot X$ is dense in $X$ and $\dot X$ is closed in $X$. In this case $\dot X=X$ and the space $X$ is discrete. By Lemma~\ref{l:sC1}, the function space $\dCF(X,Y)$ is strong Choquet.
\smallskip

To prove the ``only if'' part, assume that the function space $\dCF(X,Y)$ is strong Choquet. If $\inf Y\notin Y$, then we can apply Lemma~\ref{l:sC3} to conclude that $X$ is discrete and $Y$ is Polish. So, assume that $\inf Y\in Y$. By Proposition~\ref{p:C}, the set $\dot X$ is dense in $X$ and hence $X$ contains an isolated point $x$. By Lemma~\ref{l:sC2}, the space $Y$ is Polish. Lemma~\ref{l:SCg} implies that the set $\dot X$ is sequentially closed in $X$.
\end{proof}

\begin{example} For the Stone-\v Cech compactification $X=\beta\IN$ of the countable discrete space $\IN$ and any closed subset $Y\subset \IR$ with $\inf Y\in Y\ne\{\inf Y\}$ the function space $\dCF(X,Y)$ is strong Choquet (by Theorem~\ref{t:F}). The set $\dot X=\IN$ is sequentially closed but not closed in $X=\beta\IN$.
\end{example}

\section{Recognizing almost Polish spaces among function spaces $\dCF(X,Y)$}

Since the countable Tychonoff product of (almost) complete-metrizable spaces is (almost) complete-metrizable, we have the following simple lemma.

\begin{lemma}\label{l:AP1} If $Y\subset \IR$ is an (almost) Polish space, then for any countable discrete space $X$ the function space $\dCF(X,Y)$ is (almost) Polish.
\end{lemma}

\begin{lemma}\label{l:AP2} Let $X$ be a non-empty $T_1$-space and $Y\subset \IR$ be a non-empty subspace with $\inf Y\notin Y$. Assume that $X$ is $Y$-separated or $\dot X$ is dense in $X$. The following conditions are equivalent:
\begin{enumerate}
\item[\textup{(1)}]  $\dCF(X,Y)$ is almost complete-metrizable;
\item[\textup{(2)}]  $\dCF(X,Y)$ is almost Polish;
\item[\textup{(3)}] $Y$ is almost Polish and $X$ is countable and discrete.
\end{enumerate}
\end{lemma}

\begin{proof} The implication $(3)\Ra(2)$ was proved in Lemma~\ref{l:AP1} and $(2)\Ra(1)$ is trivial.
\smallskip

$(1)\Ra(3)$ Assume that $\dCF(X,Y)$ is almost complete-metrizable. Then it is Choquet and by Lemma~\ref{l:C3}, the space $Y$ is almost Polish and the space $X$ is discrete. In this case the Fell hypograph topology coincides with the topology of pointwise convergence and the function space $\dCF(X,Y)$ can be identified with the Tychonoff power $Y^X$. Being almost complete-metrizable, the space $Y^X$ contains a dense first-countable subspace $D$. Being regular, $Y^X$ is first-countable at each point of the set $D$. This implies that the set $X$ is countable (otherwise singletons in $Y^X$ are not $G_\delta$).
\end{proof}

\begin{lemma}\label{l:AP3} Let $Y\subset \IR$ be a subspace with $\inf Y\in Y\ne\{\inf Y\}$ and let $X$ be a $T_1$-space with dense set $\dot X$ of isolated points. The function space $\dCF(X,Y)$ is almost complete-metrizable (resp. almost Polish) if the space $Y$ is almost Polish and $X$ is a $\dot\kappa_\w$-space (with countable set $\dot X$ of isolated points).
\end{lemma}

\begin{proof} By Lemma~\ref{l:C'd}, the set $\dCF'(X,Y)$ is dense in $\dCF(X,Y)$ and by Lemma~\ref{l:F=k}, $\dCF'(X,Y)$ is homeomorphic to  the function space $C'_k(X,Y)$.

If $Y$ is almost Polish and $X$ is a $\dot\kappa_\w$-space (with countable set $\dot X$ of isolated points), then by Theorem~\ref{t:BW}(4,5), the function space $C_k'(X,Y)$ is almost complete-metrizable (resp. almost Polish) and so is its topological copy $\dCF'(X,Y)$. Then the space $\dCF(X,Y)$ is almost complete-metrizable (resp. almost Polish) since it contains a dense almost complete-metrizable (resp. almost Polish) subspace $\dCF'(X,Y)$.
 \end{proof}

\begin{lemma}\label{l:AP4} Let $Y\subset \IR$ be a subspace with $\inf Y\in Y\ne\{\inf Y\}$ and let $X$ be a $Y$-separable $T_1$-space with dense set $\dot X$ of isolated points. The function space $\dCF(X,Y)$ is almost complete-metrizable (resp. almost Polish) if and only if the space $Y$ is almost Polish and $X$ is a $\dot\kappa_\w$-space (with countable set $\dot X$ of isolated points).
\end{lemma}

\begin{proof} The ``if'' part follows from Lemma~\ref{l:AP3}. To prove the ``only if'' part, assume that the function space $\dCF(X,Y)$ is almost complete-metrizable. So, $\dCF(X,Y)$ contains a dense complete-metrizable space $D$. By Lemmas~\ref{l:C'd} and \ref{l:Gdelta}, $\dCF'(X,Y)$ is a dense $G_\delta$-set in $\dCF(X,Y)$. Then $D\cap \dCF'(X,Y)$ is a $G_\delta$-set in the completely-metrizable space $D$. Since the complement $\dCF(X,Y)\setminus \dCF'(X,Y)$ is meager in $\dCF(X,Y)$, the complement $D\setminus \dCF'(X,Y)$ is meager in $D$ by the density of $D$ in $\dCF(X,Y)$. By the Baire Theorem, the intersection $D\cap\dCF'(X,Y)$ is a dense $G_\delta$-set in $D$ and also in $\dCF'(X,Y)$. By \cite[3.11]{Ke}, the $G_\delta$-subset space $D\cap\dCF'(X,Y)$ of the completely-metrizble space $D$ is complete-metrizable, so $\dCF'(X,Y)$ is almost complete-metrizable. By Theorem~\ref{t:BW}(4), the space $Y$ is almost Polish and $X$ is a $\dot\kappa_\w$-space.

If the space $\dCF(X,Y)$ is almost Polish, then we can assume that the complete-metrizable space $D$ is Polish. Then $\dCF'(X,Y)$ is almost Polish and by Theorem~\ref{t:BW}(5), the set $\dot X$ is at most countable.
\end{proof}

The following two propositions imply Theorem~\ref{t:main}(4,5).

\begin{proposition}\label{p:AP5} Let $Y\subset \IR$ be a subspace with $\inf Y\in Y\ne\{\inf Y\}$ and $X$ be a $Y$-separable $Y$-separated space $X$. The function space $\dCF(X,Y)$ is almost complete-metrizable if and only if the space $Y$ is almost Polish and $X$ is a $\dot\kappa_\w$-space with dense set $\dot X$ of isolated points.
\end{proposition}

\begin{proof}  The ``if'' part follows from Lemma~\ref{l:AP3}. To prove the ``only if'' part, assume that the function space $\dCF(X,Y)$ is almost complete-metrizable. So, $\dCF(X,Y)$ contains a dense complete-metrizable space $D$. Then it is Baire and by Lemma~\ref{l:d'}, the set $\dot X$ is dense in $X$. By Lemma~\ref{l:AP4},  the space $Y$ is almost complete-metrizable and $X$ is a $\dot\kappa_\w$-space.
\end{proof}

By analogy we can prove

\begin{proposition}\label{p:AP6} Let $Y\subset \IR$ be a subspace with $\inf Y\in Y\ne\{\inf Y\}$, and let $X$ be a $Y$-separable $Y$-separated space. The function space $\dCF(X,Y)$ is almost Polish if and only if the space $Y$ is almost Polish and $X$ is a $\dot\kappa_\w$-space with dense and countable set $\dot X$ of isolated points.
\end{proposition}

\section{Recognizing Polish spaces among the function spaces $\dCF(X,Y)$}

In this section we recognize complete-metrizable and Polish spaces among function spaces $\dCF(X,Y)$. The case $\inf Y\notin Y$ is simple.

\begin{lemma}\label{l:P0} Let $X$ be a non-empty $T_1$-space and $Y\subset \IR$ be a non-empty subspace with $\inf Y\notin Y$. Assume that $X$ is $Y$-separated or $\dot X$ is dense in $X$. Then the following conditions are equivalent:
\begin{enumerate}
\item[\textup{(1)}]  $\dCF(X,Y)$ is complete-metrizable;
\item[\textup{(2)}]  $\dCF(X,Y)$ is Polish;
\item[\textup{(3)}] $Y$ is Polish and $X$ is countable and discrete.
\end{enumerate}
\end{lemma}

\begin{proof} The implication $(3)\Ra(2)$ trivially follows from the preservation of Polish spaces by countable Tychonoff products and $(2)\Ra(1)$ is trivial.
\smallskip

$(1)\Ra(3)$ Assume that the space $\dCF(X,Y)$ complete-metrizable. By Lemma~\ref{l:AP2}, the space $Y$ is almost Polish and the space $X$ is countable and discrete. Since $X$ is discrete, the Fell-hypograph topology on $\dCF(X,Y)$ coincides with the Tychonoff product topology on $Y^X$. The complete-metrizability of the function space $\dCF(X,Y)=Y^X$ implies the complete-metrizability of $Y$. Being almost Polish, the complete-metrizable space $Y$ is separable and hence Polish.
\end{proof}

The case $\inf Y\in Y$ is more complicated and requires some preliminary work.
We start with reminding two known definitions.

A topological space $X$
\begin{itemize}
\item is {\em hemicompact} if there exists a countable family $\K$ of compact subsets of $X$ such that each compact set $C\subset X$ is contained in some set $K\in\K$;
\item has a {\em countable network} if there exists a countable family $\mathcal N$ of subsets of $X$ such that for any open set $U\subset X$ and point $x\in U$ there exists a set $N\in\mathcal N$ such that $x\in N\subset U$.
\end{itemize}

A partial case (for $Y=\IR$ and Tychonoff $X$) the following lemma was proved by McCoy and Ntantu \cite{McN}.

\begin{lemma}\label{l:P1}  Let $Y\subset \IR$ be a subspace containing more than one point and $X$ be a $Y$-separated space.
\begin{enumerate}
\item[\textup{(1)}] If $\dCF(X,Y)$ is first-countable, then $X$ is hemicompact and $\dot X$ is countable.
\item[\textup{(2)}] If $\dCF(X,Y)$ has a countable network, then the $Y$-topology of $X$ has countable network; consequently, $\dot X$ is at most countable and $X$ is $Y$-separable.
\item[\textup{(3)}] If $\dCF(X,Y)$ has a countable network, then each compact subset of $X$ is metrizable.
\item[\textup{(4)}] If $\dCF(X,Y)$ is first-countable and has a countable network, then the space $X$ has a countable network.
\end{enumerate}
\end{lemma}

\begin{proof} Fix two real numbers $\underline{u}<\bar u$ in $Y$.
\smallskip

1h. Assume that the space $\dCF(X,Y)$ is first-countable at the constant function $\underline{c}:X\to\{\underline{u}\}\subset Y$ and fix a countable neighborhood base $\{O_n\}_{n\in\w}$ at $\underline{c}$. By the definition of the Fell hypograph topology, for every $n\in\w$ there exits a finite family $\U_n$ of non-empty open sets in $X$, a non-empty compact subset $K_n\subset X$, a function $a_n:\U_n\to\IR$,  and a real number $b_n>\underline{u}$ such that $$\underline{c}\in \lceil K_n;b_n\rceil\cap \bigcap_{U\in\U_n}\lceil U_n;a_n(U)\rfloor\subset O_n.$$
Observe that for every $U\in\U_n$ we get $a_n(U)<\sup \underline{c}(U)=\underline{u}$.
Replacing $K_n$ by a larger compact set in $X$, we can assume that $K$ has non-empty intersection with each set $U\in\U_n$.

We claim that the countable family $\{K_n\}_{n\in\w}$ witnesses that the space $X$ is hemicompact.
Given any compact subset $K\subset X$, consider the open neighborhood $\lceil K;\bar u\rceil\subset\dCF(X,Y)$ of $\underline{c}$ and find $n\in\w$ such that $O_n\subset \lceil K;\bar u\rceil$. We claim that $K\subset K_n$. Assuming that $K\not\subset K_n$, find a point $x\in K\setminus K_n$. Using Lemma~\ref{l:e}, construct a function $f:X\to Y$ such that $f(K_n)\subset\{\underline{u}\}$ and $f(x)=\bar u$.  Observe that for every $U\in\U_n$, we have $\sup f(U)\ge \sup f(U\cap K_n)\ge \underline{u}>a_n(U)$. Consequently,
$$f\in \lceil K_n;b_n\rceil\cap\bigcap_{U\in\U_n}\lceil U;a_n(U)\rfloor\subset O_n\subset \lceil K;\bar u\rceil,$$ and hence $f(x)\le\max f(K)<\bar u$, which contradicts the choice of $f$. This contradiction completes the proof of the hemicompactness of $X$.
\smallskip

1c.  Assume that the space $\dCF(X,Y)$ is first-countable at the constant function $\bar{c}:X\to\{\bar{u}\}\subset Y$ and fix a countable neighborhood base $\{O_n\}_{n\in\w}$ at $\bar{c}$. By the definition of the Fell hypograph topology, for every $n\in\w$ there exits a finite family $\U_n$ of non-empty open sets in $X$, a non-empty compact subset $K_n\subset X$, a function $a_n:\U_n\to\IR$, and a real number $b_n>\bar{u}$ such that $$\bar{c}\in \lceil K_n;b_n\rceil\cap \bigcap_{U\in\U_n}\lceil U_n;a_n(U)\rfloor\subset O_n.$$
Replacing each set $U\in\U_n$ by a suitable non-empty open subset of $U$, we can assume that either $U\subset X'$ or $U=\{x_U\}\subset \dot X$ for some isolated point $x_U$ of $X$. Replacing $K_n$ by a larger compact set, we can assume that $K$ intersects each set $U\in\U_n$. It follows that $a_n(U)<\sup\bar c(U)=\bar u$.
Let $$\dot\U_n:=\{U\in\U_n:U\cap\dot X\ne\emptyset\}=\big\{\{x\}\in \U_n:x\in\dot X\big\}.$$ We claim that $\dot X=\bigcup_{n\in\w}\bigcup\dot\U_n$.

Given any point $x\in\dot X$, consider the open neighborhood $\lceil \{x\};\underline{u}\rfloor$ of $\bar c$ in $\dCF(X,Y)$ and find $n\in\w$ such that $O_n\subset\lceil \{x\};\underline{u}\rfloor$. We claim that $\{x\}\in\U_n$.   Assuming that $\{x\}\notin \U_n$, we conclude that $x\notin\bigcup\U_n$.  Consider the function $\chi_x:X\to\{\underline{u},\bar u\}\subset Y$, defined by $\chi_x^{-1}(\underline{u})=\{x\}$. It follows from $\max_{U\in\U_n}a_n(U)<\bar u<b_n$ that $$\chi_x\in \lceil K_n;b_n\rceil\cap\bigcap_{U\in\U_n}\lceil U;a_n(U)\rfloor\subset O_n\subset \lceil \{x\};\underline{u}\rfloor,$$and hence $\chi_x(x)>\underline{u}$, which contradicts the definition of the function $\chi_x$. This contradiction shows that $\{x\}\in\dot\U_n$.
Now we see that the set $\dot X\subset\bigcup_{n\in\w}\cup\dot\U_n$ is countable (we recall that each set $\cup\dot\U_n$, $n\in\w$, is finite).
\smallskip

2. Assume that the space $\dCF(X,Y)$ has a countable network $\mathcal N$. We shall show that the $Y$-topology of $X$ has countable network. For every set $N\in\mathcal N$ consider the set $N^*:=\{x\in X:N\subset \lceil\{x\};\bar u\rceil\}$. We claim that the family $\mathcal N^*=\{N^*:N\in\mathcal N\}$ is a countable network for the $Y$-topology of the space $X$.
Fix any point $x\in X$ and its neighborhood $O_x$ in the $Y$-topology of $X$.

If the space $Y$ is connected, then the $Y$-topology coincides with the $\IR$-topology of $X$. So, we can find a continuous function $f:X\to [\underline{u},\bar{u}]\subset Y$ such that $f(x)=\underline{u}$ and $f(X\setminus O_x)\subset\{\bar u\}$.

If the space $Y$ is disconnected, then the $Y$-topology on $X$ is generated by the base consisting of clopen subsets of $X$. In this case we can find a continuous  function $f:X\to\{\underline{u},\bar u\}\subset Y$ such that $f(x)=\underline{u}$ and $f(X\setminus O_x)\subset \{\bar u\}$.

In both cases we have a continuous function  $f:X\to Y$ such that $f(x)=\underline{u}$ and $f(X\setminus O_x)\subset \{\bar u\}$.

 For the open neighborhood $\lceil\{x\};\bar u\rceil\subset \dCF(X,Y)$ of $f$, find a set $N\in\mathcal N$ such that $f\in N  \subset\lceil\{x\};\bar u\rceil$. Then $x\in N^*$ by the definition of $N^*$. On the other hand, for any $z\in X\setminus O_x$ we get $f\notin \lceil \{z\};\bar u\rceil$ and then $N\not\subset \lceil\{z\};\bar u\rceil$ and  $z\notin N^*$, which implies $x\in N^*\subset O_x$. Therefore, $\mathcal N^*$ is a countable network for $Y$-topology.  Since spaces with countable network are hereditarily separable, the space $X'$ contains a countable set $M$ such that $X'=\overline{M}^Y$, which means that the space $X$ is $Y$-separable.

Since each isolated point of $X$ remains isolated in the $Y$-topology (which has a countable network), the set $\dot X$ is at most countable.
\smallskip

3. Assume that the space $\dCF(X,Y)$ has a countable network $\mathcal N$. By the preceding statement, the $Y$-topology of $X$ has a countable network.
Since the space $X$ is $Y$-separated, on any compact subset $K$ of $X$ the $Y$-topology induces the original subspace topology of $K$, which implies that $K$ has a countable network and hence is metrizable by \cite[3.1.19]{Eng}.
\smallskip

4. If the space $\dCF(X,Y)$ is first-countable and has countable network, then $X$ is hemicomact and hence $\sigma$-compact (by the first statement). Consequently, $X$ contains a countable family $\{K_i\}_{i\in\w}$ of compact subsets such that $X=\bigcup_{i\in\w}K_i$.
By the third statement, each compact set $K_i$ has a countable network $\mathcal N_i$. Then $\mathcal N=\bigcup_{i\in\w}\mathcal N_i$ is a countable network for the space $X$.
\end{proof}

\begin{lemma}\label{l:P2} Let $Y\subset \IR$ be a subspace containing more than one point. For any  $Y$-separated space $X$, the function space $\dCF(X,Y)$ is Polish if and only if $Y$ is Polish and the space $X$ is countable and discrete.
\end{lemma}

\begin{proof} The ``if''  follows from the preservation of Polish spaces by countable Tychonoff products.
\smallskip

To prove the ``only if'' part, assume that the space $\dCF(X,Y)$ is Polish. If $\inf Y\notin Y$, then by Lemma~\ref{l:AP2}, the space $X$ is countable and discrete. Then $Y$ is Polish, being homeomorphic to a closed subset of the Polish space $\dCF(X,Y)=Y^X$.

Now assume that $\inf Y\in Y$. By Lemma~\ref{l:P1}(2), the space $X$ is $Y$-separable and the set $\dot X$ is at most countable. By Lemma~\ref{l:d'}, the set $\dot X$ is dense in $X$. Being Polish, the space $\dCF(X,Y)$ is strong Choquet. By Lemma~\ref{l:sC2}, the space $Y$ is Polish and by Lemma~\ref{l:SCg}, the set $\dot X$ is sequentially closed in $X$. By Lemma~\ref{l:P1}(4), the space $X$ has a countable network and hence it has countable tightness. Then the sequentially closed set $\dot X$ is closed in $X$, which implies that the space $X=\dot X$ is discrete and at most countable.
\end{proof}

\begin{lemma}\label{l:P2c} Let $Y\subset \IR$ be a subspace containing more than one point and $X$ be a $Y$-separable $Y$-separated $T_1$-space with dense set $\dot X$ of isolated points. The function space $\dCF(X,Y)$ is complete-metrizable if and only if $\dCF(X,Y)$ is Polish.
\end{lemma}

\begin{proof} The ``if' part is trivial. To prove the ``only if'' part, assume that the space $\dCF(X,Y)$ complete-metrizable. If $\inf Y\notin Y$, then by Lemma~\ref{l:AP2}, the space $Y$ is almost Polish and the space $X$ is countable and discrete. Then $Y$ is complete-metrizable, being homeomorphic to a closed subset of the complete-metrizable space $\dCF(X,Y)=Y^X$. Being almost Polish, the complete-metrizable space $Y$ is separable and hence Polish.

Now assume that $\inf Y\in Y$. In this case Lemma~\ref{l:AP4} implies that $X$ is a $\dot\kappa_\w$-space. By Lemma~\ref{l:P1}(1), the set $\dot X$ is countable. By Theorem~\ref{t:BW}(6), the space $C_k'(X,Y)$ is separable. By Lemmas~\ref{l:C'd} and \ref{l:F=k}, the set $\dCF'(X,Y)=C_k'(X,Y)$ is dense in $\dCF(X,Y)$, which implies that the complete-metrizable space $\dCF(X,Y)$ is separable and hence Polish.
\end{proof}

The following proposition implies Theorem~\ref{t:main}(6) announced in the introduction.

\begin{proposition}\label{p:P} Let $Y\subset \IR$ be a subspace containing more than one point. For a  non-empty $Y$-separable $Y$-separated space $X$, the following conditions are equivalent:
\begin{enumerate}
\item[\textup{(1)}] $\dCF(X,Y)$ is complete-metrizable;
\item[\textup{(2)}] $\dCF(X,Y)$ is Polish;
\item[\textup{(3)}] $Y$ is Polish and the space $X$ is countable and discrete.
\end{enumerate}
\end{proposition}

\begin{proof} The implication $(3)\Ra(2)$ trivially follows from the preservation of Polish spaces by countable Tychonoff products, and $(2)\Ra(1)$ is trivial.
\smallskip

$(1)\Ra(3)$ Assume that the function space $\dCF(X,Y)$ is complete-metrizable. If $\inf Y\notin Y$, then we can apply Lemma~\ref{l:P0} and conclude that $Y$ is Polish and $X$ is countable and discrete.

Next, consider the case $\inf Y\in Y$. Being complete-metrizable, the space $\dCF(X,Y)$ is Baire. By Lemma~\ref{l:d'}, the set $\dot X$ is dense in $X$. Now we can apply Lemmas~\ref{l:P2}, \ref{l:P2c} and conclude that the space $Y$ is Polish and $X$ is countable and discrete.
\end{proof}
\section{Recognizing function spaces $\dCF(X,Y)$ which are neither Baire nor meager}

\begin{theorem}\label{t:N} Let $Y\subset \IR$ be a Polish+meager subspace of the real line and $X$ be a non-empty $Y$-separable $Y$-separated topological space. The function space $\dCF(X,Y)$ is neither Baire nor meager if and only if one of the following conditions is satisfied:
\begin{enumerate}
\item[\textup{(1)}] $X$ is finite and  $Y$ is neither Baire nor meager;
\item[\textup{(2)}] $X$ is infinite compact, $X'^\circ=\emptyset$, $\inf Y\in Y$ and $Y$ is neither Baire nor meager;
\item[\textup{(3)}] $X$ is compact, $X'^\circ\ne\emptyset$ and $\inf Y\in \dot Y$;
\item[\textup{(4)}] $X$ is not compact, $X$ has $\DMOP$, $\inf Y\in \dot Y$, $\overline{X'^\circ}$ is compact and not empty,  and $Y$ is Baire.
\end{enumerate}
\end{theorem}

\begin{proof} First we prove that each of the conditions (1)--(4) implies that the function space $\dCF(X,Y)$ is neither meager nor Baire.
\smallskip

1. Assume that $X$ is finite and $Y$ is neither Baire nor meager. It follows that the Fell hypograph topology on $\dCF(X,Y)$ coincides with the topology of Tychonoff product $Y^X$. Taking into account that $Y$ is neither meager nor Baire, we can find an non-empty open meager subspace $M\subset Y$ and a non-empty open Baire subspace $B\subset Y$. Then $M^X$ is an open meager subspace in $Y^X=\dCF(X,Y)$, which implies that the space $\dCF(X,Y)$ is not Baire. Since the Baire space $B$ is second-countable its power $B^X$ is Baire according to \cite{Oxtoby}.
Then the space $Y^X=\dCF(X,Y)$ contains the non-empty open Baire subspace $B^X$ and hence is not meager.
\smallskip

2. Assume that the space $X$ is compact and infinite, $X'^\circ=\emptyset$, $\inf Y\in Y$ and $Y$ is neither Baire nor meager. The compactness of the space $X$ implies that $X$ has $\DMOP$. By Proposition~\ref{p:Bm2}, the space $\dCF(X,Y)$ is not Baire (otherwise $Y$ would be Baire). Since the space $Y$ is not meager, it contains a non-empty open Baire subspace $B\subset Y$. Replacing $B$ by $B\cup\{\inf Y\}$, we can assume that $\inf Y\in B$. The space $B$ is Polish+meager (being an open subspace of the Polish+meager space $Y$) and hence $B$ contains a dense Polish subspace (being Baire). By Theorem~\ref{t:BW}(2), the space $C_k'(X,B)$ is Baire. The compactness of the space $X$ ensures that $C_k'(X,B)$ is an open subspace of the space $C_k'(X,Y)$. Then the space $C_k'(X,Y)$ is not meager. By Lemma~\ref{l:F=k}, the space $C_k'(X,Y)$ is homeomorphic to $\dCF'(X,Y)$ and by Lemma~\ref{l:C'd}, the space $\dCF'(X,Y)$ is dense in the space $\dCF(X,Y)$. Then the space $\dCF(X,Y)$ is not meager since it contains a dense non-meager subspace $\dCF'(X,Y)$.
\smallskip

3. Assume that $X$ is compact, $X'^\circ\ne\emptyset$, and $\inf Y\in \dot Y$. Find a real number $\e$ such that $\{\inf Y\}=\{y\in Y:y<\e\}$ and  observe that the constant function $c:X\to\{\inf Y\}\subset Y$ is a unique point of the basic open set $\lceil X;\e\rceil$. This implies that the function space $\dCF(X,Y)$ is not meager. To see that it is not Baire, observe that the basic open set $\lceil X'^\circ,\inf Y\rfloor$ is meager (according to Lemma~\ref{l:om}).
\smallskip

4. Assume that the space $X$ is not compact, $X$ has $\DMOP$,  $\inf Y\in \dot Y$, $\overline{X'^\circ}$ is compact and not empty,  and $Y$ is Baire. The space $Y$, being Baire and Polish+meager, contains a dense Polish subspace. By Theorem~\ref{t:BW}(2) and Lemma~\ref{l:F=k}, the space $C_k'(X,Y)=\dCF'(X,Y)$  is Baire. Since $\overline{X'^\circ}$ is compact and $\inf Y\in\dot Y$, the set $$[X'^\circ,\{\inf Y\}]=[\overline{X'^\circ};\{\inf Y\}]=\dCF(X,Y)\setminus\lceil X'^\circ,\inf Y\rfloor$$ is clopen in $X$. By Lemma~\ref{l:C'd}, the Baire space $\dCF'(X,Y)$ is dense in $[{X'^\circ};\{\inf Y\}]$, which implies that the clopen subspace  $[{X'^\circ};\{\inf Y\}]$ of $\dCF(X,Y)$ is Baire and hence $\dCF(X,Y)$ is not meager. On the other hand, Lemma~\ref{l:om} ensures that its complement $\dCF(X,Y)\setminus [{X'^\circ};\{\inf Y\}]=\lceil X'^\circ,\inf Y\rfloor$ is a meager open set in $\dCF(X,Y)$, which implies that $\dCF(X,Y)$ is not Baire.
\smallskip

Now assuming that function space $\dCF(X,Y)$ is neither Baire or meager, we shall prove that one of the conditions (1)--(4) is satisfied. By Lemmas~\ref{l:MYM} and \ref{l:MXM2}, the space $Y$ is not meager and the space $X$ has $\DMOP$.

First assume that $X$ is discrete. In this case the function space $\dCF(X,Y)$ can be identified with the power $Y^X$ of $Y$. By Lemma~\ref{l:BX0}, the space $Y$ is not Baire (as $\dCF(X,Y)=Y^X$ is not Baire). By Lemma~\ref{l:MYNd}, the space $X$ is finite. So the condition (1) holds.

So, assume that $X$ is not discrete. In this case Lemma~\ref{l:MY0X2} implies that $\inf Y\in Y$. Now consider two subcases.

First we assume that $X$ is compact. If $X'^\circ=\emptyset$, then Proposition~\ref{p:Bm2} implies that Polish+meager space space $Y$ is not Baire and hence the condition (2) holds.
If $X'^\circ\ne\emptyset$, then Lemma~\ref{l:MX23} implies that $\inf Y\in \dot Y$,  which yields the condition (3).

Next, assume that $X$ is not compact. If $\dot X$ is not contained in a compact subset of $X$, then $Y$ is Baire by Lemma~\ref{l:MYN}. Being Polish+meager, the Baire space $Y$ is almost Polish. By Lemma~\ref{l:B1}(1), the set $X'$ has non-empty interior in $X$. By Lemma~\ref{l:MX23}, $\inf Y\in\dot Y$ and $\overline{X'^\circ}$ is compact. This means that the condition (4) is satisfied.

Finally, assume that the set $\dot X$ is contained in a compact subset $K$ of $X$. Since $X$ is not compact, the set $X'$ has non-empty interior.  By Lemma~\ref{l:MX23}, $\inf Y\in\dot Y$ and $\overline{X'^\circ}$ is compact. Then the space $X=\overline{\dot X}\cup \overline{X'^\circ}$ is compact, which contradicts our assumption.
\end{proof}

\section{A dichotomy for analytic function spaces $\dCF(X,Y)$}\label{s:dych}

In this section we prove Theorem~\ref{t:dycho}, announced in the introduction.

\begin{theorem}\label{t:dycho2} Let $Y\subset \IR$ be a non-empty Polish subspace with $\inf Y\notin\dot Y$. If for a $Y$-separated topological space $X$ the function space $\dCF(X,Y)$  is analytic, then $\dCF(X,Y)$ is either $\infty$-meager or $\infty$-comeager.
\end{theorem}

\begin{proof} If $Y$ is a singleton, then the function space $\dCF(X,Y)$ is a singleton. In this case $\dCF(X,Y)$ is Polish and hence $\infty$-comeager.

So, we assume that the set $Y$ contains more than two points. Being analytic, the function space $\dCF(X,Y)$ has a countable network. By Lemma~\ref{l:P1}(2), the space $X$ is $Y$-separable and the set $\dot X$ is at most countable. If the space $X$ is discrete, then $X=\dot X$ is at most countable and $\dCF(X,Y)=Y^X$ is Polish and hence $\infty$-comeager.

So, we assume that $X$ is not discrete. If $\inf Y\notin Y$, then the function space $\dCF(X,Y)$ is $\infty$-meager by Lemma~\ref{l:MY0X2}. So, we assume that $\inf Y\in Y$. If the set $X'$ has non-empty interior in $X$, then the space $\dCF(X,Y)$ is $\infty$-meager by Lemma~\ref{l:MX23} (since $\inf Y\notin\dot Y$).

So, we assume that the set $X'$ is nowhere dense in $X$. By Lemmas~\ref{l:C'd} and \ref{l:Gdelta}, the subset $\dCF'(X,Y)$ is an $\infty$-dense $G_\delta$-set in $\dCF(X,Y)$.  Being a $G_\delta$-subset of the analytic space $\dCF(X,Y)$, the space $\dCF'(X,Y)$ is analytic.  By Lemma~\ref{l:F=k}, the space $\dCF'(X,Y)$ can be identified with function space $C_k'(X,Y)$. So, the space $C_k'(X,Y)$ is analytic and by Theorem~\ref{t:dycha},  $C_k'(X,Y)$ is either Polish or $\infty$-meager. If $C'_k(X,Y)$ is Polish, then the space $\dCF(X,Y)$ contains the $\infty$-dense Polish subspace $\dCF'(X,Y)$ and hence is $\infty$-comeager.

If $C_k'(X,Y)$ is $\infty$-meager, then by Theorem~\ref{t:BW}(2), the space $X$ does not have $\DMOP$.  Applying Lemma~\ref{l:MXM1}, we conclude that the space $\dCF(X,Y)$ is $\infty$-meager.
\end{proof}

In fact, if $\inf Y\in Y$, then the analyticity of the space $\dCF(X,Y)$ in Theorem~\ref{t:dycho2} can be replaced by the analyticity of the space $C_k'(X,Y)$.

\begin{proposition}\label{p:dycho1} Let $Y\subset \IR$ be a non-empty Polish subspace with $\inf Y\in Y\setminus\dot Y$, and $X$ be a $Y$-separable $T_1$-space with dense set $\dot X$ of isolated points. If the function space  $C_k'(X,Y)$ is analytic, then $\dCF(X,Y)$ is either $\infty$-meager or $\infty$-comeager.
\end{proposition}

\begin{proof}  By Lemmas~\ref{l:C'd}, \ref{l:Gdelta} and \ref{l:F=k}, the space $\dCF'(X,Y)=C_k'(X,Y)$ is an $\infty$-dense $G_\delta$-set in $\dCF(X,Y)$.
By Theorem~\ref{t:dycha}, the analytic space $C_k'(X,Y)$ is either Polish or $\infty$-meager. If $C_k'(X,Y)$ is Polish, then the space $\dCF(X,Y)$ contains the $\infty$-dense Polish subspace $\dCF'(X,Y)$ and hence is $\infty$-comeager.

If $C_k'(X,Y)$ is $\infty$-meager, then by Theorem~\ref{t:BW}(2), the space $X$ does not have $\DMOP$.  Applying Lemma~\ref{l:MXM1}, we conclude that the space $\dCF(X,Y)$ is $\infty$-meager.
\end{proof}

\begin{proposition}\label{p:dycho2} Let $Y\subset \IR$ be a non-empty Polish subspace with $\inf Y\in Y\setminus\dot Y$, and $X$ be a $Y$-separable $Y$-separated space. If the function space  $C_k'(X,Y)$ is analytic, then $\dCF(X,Y)$ is either $\infty$-meager or $\infty$-comeager.
\end{proposition}

\begin{proof} Assume that the function space $C_k'(X,Y)$ is analytic.
If the set $\dot X$ is dense in $X$, then by  Proposition~\ref{p:dycho1}, the space $\dCF(X,Y)$ is either $\infty$-meager or $\infty$-comeager.

So, assume that $\dot X$ is not dense in $X$ and hence the set $X'^{\circ}$ is not empty. Since $\inf Y\notin\dot Y$, Lemma~\ref{l:MX23} implies that the space $\dCF(X,Y)$ is $\infty$-meager.
\end{proof}

\section{References to proofs of the statements in Table~1}\label{s:table}

In this section we provide references to lemmas that prove the statements in $8\times 7$ cells of Table 1.

\begin{center}
\begin{tabular}{|c||c|c|c|c|c|c| c|}
\hline
 & $Y_M$& $Y_{N0}$ & $Y_{N1}$ & $Y_{N2}$ & $Y_{B0}$ & $Y_{B1}$ & $Y_{B2}$\\
 \hline
 \hline
$X_{C0}$ & M \ref{l:MYM}  &\multicolumn{3}{ |c| }{N \ref{t:N}} & \multicolumn{3}{ |c| }{B \ref{l:BX0}}\\
 \hline
$X_{C1}$ &M \ref{l:MYM}   & M \ref{l:MY0X1} &\multicolumn{2}{ |c| }{N \ref{t:N}} & M \ref{l:MY0X1} & \multicolumn{2}{ |c| }{B \ref{p:Bm2}}\\
 \hline
$X_{C2}$ &M \ref{l:MYM}   & M \ref{l:MY0X2} & M \ref{l:MX23} & N \ref{t:N} & M \ref{l:MY0X2} & M \ref{l:MX23} & N \ref{t:N}\\
 \hline
$X_{B0}$ & M \ref{l:MYM} & \multicolumn{3}{ |c| }{ M \ref{l:MYN}} &\multicolumn{3}{ |c| }{B \ref{l:BX0}}\\
 \hline
$X_{B1}$ & M \ref{l:MYM} & M \ref{l:MY0X1} &\multicolumn{2}{ |c| }{ M \ref{l:MYN}} & M \ref{l:MY0X1} &\multicolumn{2}{ |c| }{B \ref{p:Bm2}}\\
 \hline
$X_{B2}$ & M \ref{l:MYM} & M \ref{l:MY0X2} & M  \ref{l:MX23} & M \ref{l:MYN} & M \ref{l:MY0X2} & M \ref{l:MX23} & N \ref{t:N}\\
 \hline
$X_{M}$ & M \ref{l:MYM} & \multicolumn{6}{ |c| }{M \ref{l:MXM2}} \\
\hline
$X_{3}$ & M \ref{l:MYM} & M \ref{l:MY0X2} &  \multicolumn{2}{ |c| }{M \ref{l:MX23}}& M \ref{l:MY0X2} & \multicolumn{2}{ |c| }{M \ref{l:MX23}}\\
\hline
 \end{tabular}
 \end{center}
 \bigskip

The equivalence of the meagerness and $\infty$-meagerness (claimed in Theorem~\ref{t:eq}) follows from the facts that in Lemmas~\ref{l:MY0X2}, \ref{l:MY0X1},  \ref{l:MX23},  \ref{l:MXM2}, \ref{l:MYM}, \ref{l:MYN} we establish the $\infty$-meagerness of the spaces $\dCF(X,Y)$ and the cells in the above table exhaust all possible cases of the interplay between the properties of the spaces $X$ and $Y$.

\end{document}